\documentclass[12pt]{article}

\usepackage{amsfonts,amsmath,amsxtra,amsthm,amssymb,tikz,latexsym}
\usepackage{mathrsfs}

\usepackage[colorlinks=true,citecolor=black,linkcolor=black,urlcolor=blue]{hyperref}

\definecolor{darkblue}{rgb}{0.0,0,0.7} 

\theoremstyle{plain}
\newtheorem{theorem}{Theorem}[section]
\newtheorem{lemma}[theorem]{Lemma}
\newtheorem{corollary}[theorem]{Corollary}
\newtheorem{proposition}[theorem]{Proposition}

\theoremstyle{definition}

\newtheorem{example}[theorem]{Example}
\theoremstyle{remark}
\newtheorem{remark}[theorem]{Remark}
\numberwithin{equation}{section}

\newcommand{\onep}{1'}
\newcommand{\twop}{2'}
\newcommand{\threep}{3'}
\newcommand{\fourp}{4'}
\newcommand{\fivep}{5'}

\newcommand{\iprime}{i'}
\newcommand{\jprime}{j'}

\newcommand{\boldtwo}{\mathbf{2}}

\newcommand{\boldtwop}{\mathbf{2'}}

\newcommand{\boldi}{\boldsymbol{i}}

\newcommand{\wt}{\mathrm{wt}}
\newcommand{\nz}{\mathrm{nz}}

\usepackage{graphics, graphicx, xcolor}
\definecolor{darkred}{rgb}{0.7,0,0} 
\newcommand{\darkred}{\color{darkred}} 
\newcommand{\defn}[1]{\emph{\darkred #1}} 

\usepackage[vcentermath, enableskew]{youngtab}
\usepackage{lipsum,  verbatim}
\usepackage{mathtools}
\usepackage{txfonts}

\title{\bf Crystal Analysis of type $C$ Stanley Symmetric Functions}

\author{Graham Hawkes \qquad Kirill Paramonov \qquad Anne Schilling\thanks{Partially supported by
NSF grant DMS--1500050.}\\
\small Department of Mathematics\\[-0.8ex]
\small University of California\\[-0.8ex]
\small One Shields Avenue\\[-0.8ex]
\small Davis, CA 95616-8633, U.S.A.\\
\small\tt \{hawkes,kirill,anne\}@math.ucdavis.edu
}

\date{
}

\begin{document}

\maketitle

\begin{abstract}
Combining results of T.K. Lam and J. Stembridge, the type $C$ Stanley symmetric function $F_w^C(\mathbf{x})$, 
indexed by an element $w$ in the type $C$ Coxeter group, has a nonnegative integer expansion in terms of Schur 
functions. We provide a crystal theoretic explanation of this fact and give an explicit combinatorial description of the 
coefficients in the Schur expansion in terms of highest weight crystal elements.

\end{abstract}

\section{Introduction}

Schubert polynomials of type $B$ and type $C$ were independently introduced by Billey and Haiman~\cite{Billey.Haiman.1995} 
and Fomin and Kirillov~\cite{Fomin.Kirillov.1996}. Stanley symmetric functions~\cite{Stanley.1984} are stable limits 
of Schubert polynomials, designed to study properties of reduced words of Coxeter group elements. In his Ph.D. thesis,
T.K. Lam~\cite{Lam.1995} studied properties of Stanley symmetric functions of types $B$ (and similarly $C$) and $D$.
In particular he showed, using Kra\'skiewicz insertion~\cite{Kraskiewicz.1989,Kraskiewicz.1995}, that the type $B$ Stanley 
symmetric functions have a positive integer expansion in terms of $P$-Schur functions. On the other hand, 
Stembridge~\cite{Stembridge.1989} proved that the $P$-Schur functions expand positively in terms of Schur functions. 
Combining these two results, it follows that Stanley symmetric functions of type $B$ (and similarly type $C$)
have a positive integer expansion in terms of Schur functions.

Schur functions $s_\lambda(\mathbf{x})$, indexed by partitions $\lambda$, are ubiquitous in combinatorics and 
representation theory. They are the characters of the symmetric group and can also be interpreted as characters of type 
$A$ crystals. In~\cite{Morse.Schilling.2016}, this was exploited to provide a combinatorial interpretation in terms of 
highest weight crystal elements of the coefficients in the Schur expansion of Stanley symmetric functions in type $A$.
In this paper, we carry out a crystal analysis of the Stanley symmetric functions $F_w^C(\mathbf{x})$ of type $C$, indexed
by a Coxeter group element $w$. In particular, we use Kra\'skiewicz insertion~\cite{Kraskiewicz.1989,Kraskiewicz.1995}
and Haiman's mixed insertion~\cite{Haiman.1989} to find a crystal structure on primed tableaux, which in turn implies
a crystal structure $\mathcal{B}_w$ on signed unimodal factorizations of $w$ for which $F^C_w(\mathbf{x})$ is a character.
Moreover, we present a type $A$ crystal isomorphism $\Phi \colon \mathcal{B}_w \rightarrow \bigoplus_\lambda 
\mathcal{B}_{\lambda}^{\oplus g_{w\lambda}}$ for 
some combinatorially defined nonnegative integer coefficients $g_{w\lambda}$; here $\mathcal{B}_\lambda$ is the
type $A$ highest weight crystal of highest weight $\lambda$ . This implies the desired decomposition 
$F^C_w(\mathbf{x}) = \sum_\lambda g_{w\lambda} s_\lambda (\mathbf{x})$ (see Corollary~\ref{corollary.main2})
and similarly for type $B$.

The paper is structured as follows. In Section~\ref{section.background}, we review type $C$ Stanley symmetric functions 
and type $A$ crystals. In Section~\ref{section.isomorphism} we describe our crystal isomorphism by combining a slight
generalization of the Kra\'skiewicz insertion~\cite{Kraskiewicz.1989,Kraskiewicz.1995} and Haiman's mixed 
insertion~\cite{Haiman.1989}. The main result regarding the crystal structure under Haiman's mixed insertion 
is stated in Theorem~\ref{theorem.main2}. The combinatorial interpretation of the coefficients $g_{w\lambda}$ 
is given in Corollary~\ref{corollary.main2}. In Section~\ref{section.semistandard}, we provide an alternative
interpretation of the coefficients $g_{w\lambda}$ in terms of semistandard unimodal tableaux.
Appendices~\ref{section.proof main2} and~\ref{section.proof main3} are reserved for the proofs of
Theorems~\ref{theorem.main2} and~\ref{theorem.main3}.

\subsection*{Acknowledgments}
We thank the anonymous referee for pointing out reference~\cite{Liu.2017} and furthermore the connections between 
our crystal operators and those obtained by intertwining crystal operators on words with Haiman's symmetrization of 
shifted mixed insertion~\cite[Section 5]{Haiman.1989} and the conversion map~\cite[Proposition~14]{SW.2001} as outlined 
in Remark~\ref{remark.doubling}. We thank Toya Hiroshima for pointing out that the definition of the reading word of
a primed tableau was misleading in a previous version of this paper.

\section{Background}
\label{section.background}

\subsection{Type $C$ Stanley symmetric functions}

The \defn{Coxeter group} $W_C$ of type $C_n$ (or type $B_n$), also known as the hyperoctahedral group or the group 
of signed permutations, is a finite group generated by $\{s_0, s_1, \ldots, s_{n-1}\}$ subject to the quadratic relations
$s_i^2 = 1$ for all $i \in I = \{0,1,\ldots,n-1\}$, the commutation relations $s_i s_j = s_j s_i$ provided $|i-j|>1$, and the
braid relations $s_i s_{i+1} s_i = s_{i+1} s_i s_{i+1}$ for all $i>0$ and $s_0 s_1 s_0 s_1 = s_1 s_0 s_1 s_0$.

It is often convenient to write down an element of a Coxeter group as a sequence of indices of $s_i$ in the product 
representation of the element. For example, the element $w = s_2 s_1 s_2 s_1 s_0 s_1 s_0 s_1$ is represented by the 
word ${\bf w} = 2120101$. A word of shortest length $\ell$ is referred to as a \defn{reduced word} and $\ell(w):=\ell$ 
is referred as the length of $w$. The set of all reduced words of the element $w$ is denoted by $R(w)$.

\begin{example}
The set of reduced words for $w = s_2 s_1 s_2 s_0 s_1 s_0$ 
is given by
$$R(w) = \{ 210210, 212010, 121010, 120101, 102101 \}.$$
\end{example}

We say that a reduced word $a_1 a_2 \ldots a_\ell$ is \defn{unimodal} if there exists an index $v$, such that 
$$a_1 > a_2 > \cdots > a_v < a_{v+1} < \cdots < a_\ell.$$

Consider a reduced word $\textbf{a} = a_1 a_2 \ldots a_{\ell(w)}$ of a Coxeter group element $w$. 
A \defn{unimodal factorization} of $\textbf{a}$ is a factorization 
$\mathbf{A} = (a_1 \ldots a_{\ell_1}) (a_{\ell_1+1} \ldots a_{\ell_2}) \cdots (a_{\ell_r + 1} \ldots a_L)$ such that each factor 
$(a_{\ell_i+1} \ldots a_{\ell_{i+1}})$ is unimodal. Factors can be empty.

For a fixed Coxeter group element $w$, consider all reduced words $R(w)$, and denote the set of all unimodal 
factorizations for reduced words in $R(w)$ as $U(w)$. 
Given a factorization $\mathbf{A} \in U(w)$, define the \defn{weight} of a factorization $\wt(\mathbf{A})$ to be the vector 
consisting of the number of elements in each factor. Denote by $\nz(\mathbf{A})$ the number of non-empty factors of 
$\mathbf{A}$. 

\begin{example}
For the factorization $\mathbf{A} = (2102)()(10) \in U(s_2 s_1 s_2 s_0 s_1 s_0)$, we have $\wt(\mathbf{A}) = (4,0,2)$ 
and $\nz(\mathbf{A}) = 2$.
\end{example}

Following~\cite{Billey.Haiman.1995, Fomin.Kirillov.1996, Lam.1995}, the \defn{type $C$ Stanley symmetric function} 
associated to $w\in W_C$ is defined as
\begin{equation}
\label{equation.StanleyC}
	F^C_w(\mathbf{x}) = \sum_{\mathbf{A} \in U(w)} 2^{\nz(\mathbf{A})} 
        \mathbf{x}^{\wt(\mathbf{A})}.
\end{equation}
Here $\mathbf{x} = (x_1, x_2, x_3, \ldots)$ and $\mathbf{x}^{\mathbf{v}} = x_1^{v_1} x_2^{v_2} x_3^{v_3} \cdots$.
It is not obvious from the definition why the above functions are symmetric. We refer reader to~\cite{Billey.2014},
where this fact follows easily from an alternative definition.
 
\defn{Type $B$ Stanley symmetric functions} are also labeled by $w\in W_C$ (as the type $B$ and $C$ Coxeter
groups coincide) and differ from $F_w^C(w)$ by an overall factor $2^{-o(w)}$
\[
	F_w^B(\mathbf{x}) = 2^{-o(w)} F_w^C(\mathbf{x}),
\] 
where $o(w)$ is the number of zeroes in a reduced word for $w$. Loosely speaking, our combinatorial interpretation 
in the type $C$ case respects this power of 2 -- that is, we will  get a valid combinatorial interpretation in the type $B$ 
case by dividing by $2^{o(w)}$. 

\subsection{Type $A$ crystal of words}

Crystal bases~\cite{kashiwara.1994} play an important role in many areas of mathematics. For example, they make it 
possible to analyze representation theoretic questions using combinatorial tools. Here we only review the crystal of words 
in type $A_n$ and refer the reader for more background on crystals to~\cite{Bump.Schilling.2017}.

Consider the set of words $\mathcal{B}_n^h$ of length $h$ in the alphabet $\{1,2,\ldots,n+1\}$.
We impose a crystal structure on $\mathcal{B}_n^h$ by defining lowering operators $f_i$ and raising operators $e_i$
for $1\leqslant i \leqslant n$ and a weight function. The weight of $\mathbf{b} \in \mathcal{B}_n^h$ is the tuple
$\wt(\mathbf{b}) = (a_1,\ldots, a_{n+1})$, where $a_i$ is the number of letters $i$ in $\mathbf{b}$.
The crystal operators $f_i$ and $e_i$ only depend on the letters $i$ and $i+1$ in $\mathbf{b}$.
Consider the subword $\mathbf{b}^{\{i,i+1\}}$ of $\mathbf{b}$ consisting only of the letters $i$ and $i+1$.
Successively bracket any adjacent pairs $(i+1) i$ and remove these pairs from the word. The resulting word is
of the form $i^a (i+1)^b$ with $a,b\geqslant 0$. Then $f_i$ changes this subword within $\mathbf{b}$ to $i^{a-1} (i+1)^{b+1}$
if $a>0$ leaving all other letters unchanged and otherwise annihilates $\mathbf{b}$. The operator $e_i$ changes this 
subword within $\mathbf{b}$ to $i^{a+1} (i+1)^{b-1}$ if $b>0$ leaving all other letters unchanged and otherwise annihilates 
$\mathbf{b}$. 

We call an element $\mathbf{b}\in \mathcal{B}_n^h$ \defn{highest weight} if $e_i(\mathbf{b})=\mathbf{0}$ for all
$1\leqslant i\leqslant n$ (meaning that all $e_i$ annihilate $\mathbf{b}$).

 \begin{theorem} \cite{Kashiwara.Nakashima.1994}
 A word $\mathbf{b} = b_1 \ldots b_h \in \mathcal{B}_n^h$ is highest weight if and only if it is a Yamanouchi word.
 That is, for any index $k$ with $1 \leqslant k \leqslant h$ the weight of a subword $b_k b_{k+1} \ldots b_h$ is a partition.
 \end{theorem}

\begin{example}
The word $85744234654333222211111$ is highest weight.
\end{example}

Two crystals $\mathcal{B}$ and $\mathcal{C}$ are said to be \defn{isomorphic} if there exists a bijective map 
$\Phi \colon \mathcal{B} \rightarrow \mathcal{C}$ that preserves the weight function and commutes with the crystal 
operators $e_i$ and $f_i$. A \defn{connected component} $X$ of a crystal is a set of elements where for any two
$\mathbf{b},\mathbf{c} \in X$ one can reach $\mathbf{c}$ from $\mathbf{b}$ by applying a sequence of $f_i$ and $e_i$.

\begin{theorem} \cite{Kashiwara.Nakashima.1994}
Each connected component of $\mathcal{B}_n^h$ has a unique highest weight element.
Furthermore, if $\mathbf{b}, \mathbf{c} \in \mathcal{B}_n^h$ are highest weight elements such that
$\wt(\mathbf{b}) = \wt(\mathbf{c})$, then the connected components generated by $\mathbf{b}$ and 
$\mathbf{c}$ are isomorphic.
\end{theorem}

We denote a connected component with a highest weight element of highest weight $\lambda$ by
$\mathcal{B}_\lambda$. The \defn{character} of the crystal $\mathcal{B}$ is defined to be a polynomial
in the variables $\mathbf{x}=(x_1,x_2,\ldots,x_{n+1})$
$$\chi_{\mathcal{B}} (\mathbf{x}) = \sum_{\mathbf{b} \in \mathcal{B}} \mathbf{x}^{\wt(\mathbf{b})}.$$

\begin{theorem}[\cite{Kashiwara.Nakashima.1994}]
The character of $\mathcal{B}_{\lambda}$ is equal to the Schur polynomial $s_\lambda (\mathbf{x})$ (or Schur 
function in the limit $n\to \infty$).
\end{theorem}

\section{Crystal isomorphism} 
\label{section.isomorphism}

In this section, we combine a slight generalization of the Kra\'skiewicz insertion, reviewed in 
Section~\ref{section.kraskiewicz}, and Haiman's mixed insertion, reviewed in Section~\ref{section.implicit}, 
to provide an isomorphism of crystals between the crystal of words $\mathcal{B}^h$ and certain sets of primed tableaux. 
Our main result of this section is stated in Theorem~\ref{theorem.main0}, which asserts that the recording tableaux under 
the mixed insertion is constant on connected components of $\mathcal{B}^h$.

\subsection{Kra\'skiewicz insertion}
\label{section.kraskiewicz}

In this section, we describe the Kra\'skiewicz insertion. To do so, we first need to define 
the \defn{Edelman--Greene insertion}~\cite{Edelmann.Greene.1987}. It is defined for a word 
$\mathbf{w} = w_1 \ldots w_\ell$ and a letter $k$ such that the concatenation $w_1 \ldots w_\ell k$ is an $A$-type 
reduced word. The Edelman--Greene insertion of a letter $k$ into an {\it increasing} word 
$\mathbf{w} = w_1 \ldots w_\ell$, denoted by $\mathbf{w} \leftsquigarrow k$, is constructed as follows:
\begin{enumerate}
	\item If $w_\ell < k$, then $\mathbf{w} \leftsquigarrow k = \mathbf{w'},$ where 
	$\mathbf{w'} = w_1 w_2 \ldots w_\ell\ k$.
	\item If $k>0$ and $k\, k+1 = w_i \, w_{i+1}$ for some $1\leqslant i < \ell$, then 
	$\mathbf{w} \leftsquigarrow k = k+1  \leftsquigarrow\mathbf{w}$. 
	\item Else let $w_i$ be the leftmost letter in $\mathbf{w}$ such that $w_i>k$. Then
	$\mathbf{w} \leftsquigarrow k = w_i \leftsquigarrow \mathbf{w'}$, where 
	$\mathbf{w'} = w_1 \ldots w_{i-1}\ k\ w_{i+1} \ldots w_\ell$.
\end{enumerate}
In the cases above, when $\mathbf{w} \leftsquigarrow k = k'  \leftsquigarrow\mathbf{w'}$,
the symbol $k'  \leftsquigarrow\mathbf{w'}$ indicates a word $\mathbf{w'}$ together with a ``bumped''
letter $k'$.

Next we consider a reduced unimodal word $\mathbf{a} = a_1 a_2 \ldots a_\ell$ with 
$a_1 > a_2 >\cdots > a_v < a_{v+1} < \cdots < a_\ell$. 
The \defn{Kra\'skiewicz row insertion} \cite{Kraskiewicz.1989,Kraskiewicz.1995} is defined for a unimodal word 
$\mathbf{a}$ and a letter $k$ such that the concatenation $a_1 a_2 \ldots a_\ell k$ is a $C$-type reduced word. 
The Kra\'skiewicz row insertion of $k$ into $\mathbf{a}$ (denoted similarly as $\mathbf{a} \leftsquigarrow k$), is
performed as follows:
\begin{enumerate}
	\item If $k=0$ and there is a subword $101$ in $\mathbf{a}$, then $\mathbf{a} \leftsquigarrow 0 = 
	0 \leftsquigarrow \mathbf{a}$.
	\item If $k \neq 0$ or there is no subword $101$ in $\mathbf{a}$, denote the decreasing part $a_1 \ldots a_v$ 
	as $\mathbf{d}$ and the increasing part $a_{v+1} \ldots a_\ell$ as $\mathbf{g}$. Perform the Edelman-Greene 
	insertion of $k$ into $\mathbf{g}$.
	\begin{enumerate}
		\item If $a_\ell < k$, then $\mathbf{g} \leftsquigarrow k = a_{v+1} \ldots a_\ell k =: \mathbf{g'}$ and
		$\mathbf{a} \leftsquigarrow k = \mathbf{d} \mathbf{g} \leftsquigarrow k = \mathbf{d\ g'} =: \mathbf{a'}$.
		\item If there is a bumped letter and $\mathbf{g} \leftsquigarrow k = k' \leftsquigarrow \mathbf{g'}$, negate all 
		the letters in $\mathbf{d}$ (call the resulting word $-\mathbf{d}$) and perform the Edelman-Greene insertion 
		$-\mathbf{d} \leftsquigarrow -k'$. Note that there will always be a bumped letter, and so 
		$-\mathbf{d} \leftsquigarrow -k' = -k'' \leftsquigarrow -\mathbf{d'}$ for some decreasing word $\mathbf{d'}$. 
		The result of the Kra\'skiewicz insertion is: $\mathbf{a} \leftsquigarrow k = \mathbf{d}[\mathbf{g} \leftsquigarrow k] 
		= \mathbf{d}[k' \leftsquigarrow \mathbf{g'}] =  - [\mathbf{-d} \leftsquigarrow -k']\ \mathbf{g'} =
		[k'' \leftsquigarrow \mathbf{d'}]\mathbf{g'} = k'' \leftsquigarrow \mathbf{a'}$, where $\mathbf{a'} := \mathbf{d'g'}$.
	\end{enumerate}
\end{enumerate}

\begin{example}
\begin{equation*}
	31012 \leftsquigarrow 0 =0 \leftsquigarrow 31012, \quad 3012 \leftsquigarrow 0 = 0 \leftsquigarrow 3102,
\end{equation*}
\begin{equation*}
	31012 \leftsquigarrow 1 = 1 \leftsquigarrow 32012, \quad 31012 \leftsquigarrow 3 = 310123.
\end{equation*}
\end{example}

The insertion is constructed to ``commute'' a unimodal word with a letter: If 
$\mathbf{a} \leftsquigarrow k = k' \leftsquigarrow \mathbf{a'}$, the two elements of the type $C$ Coxeter group 
corresponding to concatenated words $\mathbf{a}\ k$ and $k' \mathbf{a'}$ are the same.

The type $C$ Stanley symmetric functions~\eqref{equation.StanleyC} are defined in terms of unimodal factorizations.
To put the formula on a completely combinatorial footing, we need to treat the powers of $2$ by introducing signed
unimodal factorizations. A \defn{signed unimodal factorization} of $w\in W_C$ is a unimodal factorization 
$\mathbf{A}$ of $w$, in which every non-empty factor is assigned either a $+$ or $-$ sign. Denote the set of all 
signed unimodal factorizations of $w$ by $U^{\pm} (w)$.

For a signed unimodal factorization $\mathbf{A} \in U^{\pm} (w)$, define $\wt(\mathbf{A})$ to be the vector with $i$-th 
coordinate equal to the number of letters in the $i$-th factor of $\mathbf{A}$. Notice from~\eqref{equation.StanleyC} that
\begin{equation}
\label{equation.Upm}
	F^C_{w}(\mathbf{x}) = \sum_{\mathbf{A} \in U^{\pm}(w)} \mathbf{x}^{\wt(\mathbf{A})}.
\end{equation}

We will use the Kra\'skiewicz insertion to construct a map between signed unimodal factorizations of a Coxeter group 
element $w$ and pairs of certain types of tableaux $(\mathbf{P},\mathbf{T})$. We define these types of tableaux next.

A \defn{shifted diagram} $\mathcal{S}(\lambda)$ associated to a partition $\lambda$ with distinct parts is the set of 
boxes in positions $\{(i,j) \mid \ 1\leqslant i\leqslant \ell(\lambda), \ i\leqslant j\leqslant \lambda_i+i-1\}$. Here, we use 
English notation, where the box $(1,1)$ is always top-left.

Let $X^\circ_n$ be an ordered alphabet of $n$ letters $X^\circ_n = \{0< 1 < 2< \cdots < n-1\}$, and let $X'_n$ be an 
ordered alphabet of $n$ letters together with their primed counterparts as $X'_n = \{1' < 1 < 2'< 2< \cdots <n' < n\}$.

Let $\lambda$ be a partition with distinct parts.
A \defn{unimodal tableau} $\mathbf{P}$ of shape $\lambda$ on $n$ letters is a filling of $\mathcal{S}(\lambda)$ with 
letters from the alphabet $X^\circ_n$ such that the word $P_i$ obtained by reading the $i$th row from the top
of $\mathbf{P}$ from left to right, is a unimodal word, and $P_i$ is the longest unimodal subword in the concatenated 
word $P_{i+1} P_i$ \cite{Billey.2014} (cf. also with decomposition tableaux~\cite{Serrano.2010,Cho.2013}). 
The \defn{reading word} of a unimodal tableau $\mathbf{P}$ is given by $\pi_{\mathbf{P}} = P_\ell P_{\ell-1} \ldots P_1$. 
A unimodal tableau is called \textit{reduced} if $\pi_{\mathbf{P}}$ is a type $C$ reduced word corresponding to the Coxeter 
group element $w_{\mathbf{P}}$. Given a fixed Coxeter group element $w$, denote the set of reduced unimodal tableaux 
$\mathbf{P}$ of shape $\lambda$ with $w_{\mathbf{P}} = w$ as $\mathcal{UT}_w (\lambda)$.

A \defn{signed primed tableau} $\mathbf{T}$ of shape $\lambda$ on $n$ letters 
(cf. semistandard $Q$-tableau~\cite{Lam.1995}) is a filling of $\mathcal{S}(\lambda)$ 
with letters from the alphabet $X'_n$ such that:
\begin{enumerate}
	\item The entries are weakly increasing along each column and each row of $\mathbf{T}$.
	\item Each row contains at most one $i'$ for every $i = 1,\ldots,n$.
	\item Each column contains at most one $i$ for every $i = 1,\ldots,n$.
\end{enumerate}
The reason for using the word ``signed'' in the name is to distinguish the set of primed tableaux above from the 
``unsigned" version described later in the chapter.

Denote the set of signed primed tableaux of shape $\lambda$ by $\mathcal{PT^{\pm}} (\lambda)$. Given an element 
$\mathbf{T} \in \mathcal{PT^{\pm}} (\lambda)$, define the weight of the tableau $\wt(\mathbf{T})$ as the vector with
$i$-th coordinate equal to the total number of letters in $\mathbf{T}$ that are either $i$ or $i'$.

\begin{example}
$\Bigg(\young(43201,:212,::0),\ \young(11\twop\threep3,:\twop2\threep,::4)\Bigg)$ is a pair consisting of a unimodal 
tableau and a signed primed tableau both of shape $(5,3,1)$.
\end{example}

For a reduced unimodal tableau $\mathbf{P}$ with rows $P_\ell, P_{\ell-1}, \ldots, P_1$, the Kra\'skiewicz insertion of 
a letter $k$ into tableau $\mathbf{P}$ (denoted again by $\mathbf{P} \leftsquigarrow k$) is performed as follows:
\begin{enumerate}
	\item Perform Kra\'skiewicz insertion of the letter $k$ into the unimodal word $P_1$. If there is no bumped letter and 
	$P_1 \leftsquigarrow k = P'_1$, the algorithm terminates and the new tableau $\mathbf{P'}$ consists of rows 
	$P_\ell, P_{\ell-1}, \ldots, P_2, P'_1$. If there is a bumped letter and $P_1 \leftsquigarrow k  = k' \leftsquigarrow P'_1$,
	continue the algorithm by inserting $k'$ into the unimodal word $P_2$.
	\item Repeat the previous step for the rows of $\mathbf{P}$ until either the algorithm terminates, in which case the new 
	tableau $\mathbf{P}'$ consists of rows $P_\ell, \ldots, P_{s+1}, P'_s, \ldots, P'_1$, or, 
        the insertion continues until we bump a letter $k_e$ from $P_\ell$, in which case we then put $k_e$ on a new row of 
	the shifted shape of $\mathbf{P'}$, so that the resulting tableau $\mathbf{P'}$ consists of rows 
	$k_e, P'_\ell, \ldots, P'_1$. 
\end{enumerate}

\begin{example}
$$\young(43201,:212,::0) \leftsquigarrow 0 = \young(43210,:210,::01),$$ 
since the insertions row by row are given by $43201 \leftsquigarrow 0 =0 \leftsquigarrow 43210$,
$212 \leftsquigarrow 0 =  1 \leftsquigarrow 210$, and $0 \leftsquigarrow 1 = 01$.
\end{example}

\begin{lemma} \cite{Kraskiewicz.1989}
Let $\mathbf{P}$ be a reduced unimodal tableau with reading word $\pi_\mathbf{P}$ for an element $w\in W_C$.
Let $k$ be a letter such that $\pi_\mathbf{P}k$ is a reduced word. Then the tableau 
$\mathbf{P'} = \mathbf{P} \leftsquigarrow k$ is a reduced unimodal tableau, for which the reading word $\pi_{\mathbf{P'}}$ 
is a reduced word for $w s_k$.
\end{lemma}

\begin{lemma} \cite[Lemma 3.17]{Lam.1995}
 \label{lemma.ins}
 Let $\mathbf{P}$ be a unimodal tableau, and $\mathbf{a}$ a unimodal word such that $\pi_{\mathbf{P}}\mathbf{a}$ is reduced. 
Let $(x_1,y_1), \ldots, (x_r, y_r)$ be the (ordered) list of boxes added when  $\mathbf{P} \leftsquigarrow {\mathbf{a}}$ 
is computed. Then there exists an index $v$, such that $x_1 < \cdots < x_v \geqslant \cdots \geqslant x_r $ and 
$y_1 \geqslant \cdots \geqslant y_v < \cdots < y_r$.
\end{lemma}

Let $\mathbf{A} \in U^{\pm} (w)$ be a signed unimodal factorization with unimodal factors 
$\mathbf{a}_1, \mathbf{a}_2, \ldots, \mathbf{a}_n$. We recursively construct a sequence 
$(\emptyset, \emptyset) = (\mathbf{P}_0, \mathbf{T}_0),\ (\mathbf{P}_1, \mathbf{T}_1), \ldots, (\mathbf{P}_n, \mathbf{T}_n) 
= (\mathbf{P}, \mathbf{T})$ of tableaux, where 
$\mathbf{P}_s \in \mathcal{UT}_{(\mathbf{a}_1 \mathbf{a}_2 \ldots \mathbf{a}_s)} (\lambda^{(s)})$ and $\mathbf{T}_s 
\in \mathcal{PT}^{\pm} (\lambda^{(s)})$ are tableaux of the same shifted shape $\lambda^{(s)}$.

To obtain the \defn{insertion tableau} $\mathbf{P}_s$, insert the letters of $\mathbf{a}_s$ one by one from left to right, into 
$\mathbf{P}_{s-1}$. Denote the shifted shape of $\mathbf{P}_{s}$ by $\lambda^{(s)}$. Enumerate the boxes in 
the skew shape $\lambda^{(s)} / \lambda^{(s-1)}$ in the order they appear in $\mathbf{P}_s$. Let these boxes be 
$(x_1,y_1), \ldots, (x_{\ell_s}, y_{\ell_s})$.

Let $v$ be the index that is guaranteed to exist by Lemma~\ref{lemma.ins} when we compute 
$\mathbf{P_{s-1}} \leftsquigarrow {\mathbf{a_s}}$. The \defn{recording tableau} $\mathbf{T}_{s}$ is a primed 
tableau obtained from $\mathbf{T}_{s-1}$ by adding the boxes $(x_1, y_1), \ldots, (x_{v-1}, y_{v-1})$, each filled with 
the letter $s'$, and the boxes $(x_{v+1},y_{v+1}), \ldots, (x_{\ell_s},  y_{\ell_s})$, each filled with the letter $s$. 
The special case is the box $(x_v,y_v)$, which could contain either $s'$ or $s$. 
The letter is determined by the sign of the factor $\mathbf{a}_s$: If the sign is $-$, the box is filled with the letter $s'$, and if 
the sign is $+$, the box is filled with the letter $s$.
We call the resulting map the \defn{primed Kra\'skiewicz map} $\mathrm{KR}'$.

\begin{example}
Given a signed unimodal factorization $\mathbf{A} = (-0) (+212) (-43201)$, the sequence of tableaux is
$$ (\emptyset,\emptyset), \quad (\ \young(0),\young(\onep)\ ), \quad \Big( \ \young(212,:0), \young(\onep\twop2,:2)\ \Big), 
\quad \Bigg(\ \young(43201,:212,::0),\young(\onep\twop2\threep3,:2\threep3,::\threep)\ \Bigg). $$
\end{example}

If the recording tableau is constructed, instead, by simply labeling its boxes with $1,2,3,\ldots$ in the order these boxes 
appear in the insertion tableau, we recover the original Kra\'skiewicz map  \cite{Kraskiewicz.1989,Kraskiewicz.1995}, 
which is a bijection
\begin{equation*}
	\mathrm{KR}\colon R(w) \rightarrow 
	\bigcup_{\lambda} \big[\mathcal{UT}_w (\lambda) \times \mathcal{ST} (\lambda)\big],
\end{equation*}
 where $\mathcal{ST}(\lambda)$ is the set of \defn{standard shifted tableau} of shape $\lambda$, i.e., the set of fillings of $\mathcal{S} (\lambda)$ with letters
$1,2, \ldots,|\lambda|$ such that each letter appears exactly once, each row filling is increasing, and each 
column filling is increasing. 

\begin{theorem}
\label{theorem.KR}
The primed Kra\'skiewicz map is a bijection
\begin{equation*}
	\mathrm{KR}'\colon U^{\pm}(w) \rightarrow 
	\bigcup_{\lambda} \big[\mathcal{UT}_w (\lambda) \times \mathcal{PT}^{\pm} (\lambda)\big].
\end{equation*}
\end{theorem}

\begin{proof}
First we show that the map is well-defined:  Let $\mathbf{A} \in U^{\pm}(w)$ such that 
$\mathrm{KR}'(A) = (\mathbf{P}, \mathbf{Q})$.  The fact that  $\mathbf{P}$ is a unimodal tableau follows from the fact 
that $\mathrm{KR}$ is well-defined. On the other hand, $\mathbf{Q}$ satisfies Condition (1) in the definition of signed 
primed tableaux since its entries are weakly increasing with respect to the order the associated boxes are added to 
$\mathbf{P}$. Now fix an $s$ and consider the insertion $\mathbf{P_{s-1}} \leftsquigarrow {\mathbf{a_s}}$. Refer to the 
set-up in Lemma~\ref{lemma.ins}. Then, $y_1<\cdots<y_v$ implies there is at most one $s'$ in each row and 
$y_v\geqslant \cdots \geqslant y_{\ell_s}$ implies there is at most one $s$ in each column, so Conditions (2) and (3) of 
the definition have been verified, implying that indeed $\mathbf{Q}$ is a signed primed tableau.

Now suppose $(\mathbf{P},\mathbf{Q}) \in \bigcup_{\lambda} \big[\mathcal{UT}_w (\lambda) \times 
\mathcal{PT}^{\pm} (\lambda)\big]$.  The ordering of the alphabet $X'$ induces a partial order on the set of boxes of 
$\mathbf{Q}$.  Refine this ordering as follows:  Among boxes containing an $s'$, box $b$ is greater than box $c$ if box 
$b$ lies below box $c$.  Among boxes containing an $s$, box $b$ is greater than box $c$ if box $b$ lies to the right of 
box $c$.  Let the standard shifted tableau induced by the resulting total order be denoted $\mathbf{Q}^*$.  

Let $w=\mathrm{KR}^{-1}(\mathbf{P},\mathbf{Q}^*)$. Divide $w$ into factors, where the size of the $s$-th factor is equal 
to the $s$-th entry in $\wt(\mathbf{Q})$. Let $\textbf{A} =\textbf{a}_1 \ldots \textbf{a}_n$ be the 
resulting factorization, where the sign of $\mathbf{a}_s$ is determined as follows: Consider the lowest leftmost box in 
$\mathbf{Q}$ that contains an $s$ or $s'$ (such a box must exist if $\textbf{a}_s \neq \emptyset$).  If this box contains 
an $s$ give $\mathbf{a}_s$ a positive sign, and otherwise a negative sign. Let $b_1,\ldots, b_{|\textbf{a}_s|}$ denote the 
boxes of $\mathbf{Q}^*$ corresponding to $\textbf{a}_s$ under $\mathrm{KR}^{-1}$.  The construction of $\mathbf{Q}^*$ and 
the fact that $\mathbf{Q}$ is a primed shifted tableau imply that the coordinates of these boxes satisfy the hypothesis 
of Lemma \ref{lemma.ins}.  Since these are exactly the boxes that appear when we compute 
$\mathbf{P_{s-1}} \leftsquigarrow \mathbf{a}_s$, Lemma \ref{lemma.ins} implies that $\mathbf{a}_s$ is unimodal. 
It follows that $\mathbf{A}$ is a signed unimodal factorization mapping to $(\mathbf{P},\mathbf{Q})$ under $\mathrm{KR}'$.
It is not hard to see $\mathbf{A}$ is unique. 
\end{proof}

Theorem~\ref{theorem.KR} and Equation~\eqref{equation.Upm} imply the following relation:
\begin{equation}
\label{equation.PTpm}
	F^C_{w}(\mathbf{x}) = \sum_{\lambda} \big|\mathcal{UT}_w (\lambda) \big| \sum_{\mathbf{T} \in 
	\mathcal{PT}^{\pm}(\lambda)} \mathbf{x}^{\wt(\mathbf{T})}.
\end{equation}

\begin{remark}
The sum $\sum_{\mathbf{T} \in \mathcal{PT}^{\pm}(\lambda)} \mathbf{x}^{\wt(\mathbf{T})}$ is also known as the
$Q$-Schur function. The expansion~\eqref{equation.PTpm}, with a slightly different interpretation of $Q$-Schur function, 
was shown in~\cite{Billey.Haiman.1995}.
\end{remark}

At this point, we are halfway there to expand $F^C_{w}(\mathbf{x})$ in terms of Schur functions.
In the next section we introduce a crystal structure on the set $\mathcal{PT} (\lambda)$ of unsigned primed tableaux.

\subsection{Mixed insertion}
\label{section.implicit}

Set $\mathcal{B}^h = \mathcal{B}^h_{\infty}$. Similar to the well-known RSK-algorithm, mixed insertion~\cite{Haiman.1989} 
gives a bijection between $\mathcal{B}^h$ and the set of pairs of tableaux $(\mathbf{T}, \mathbf{Q})$, but in this case 
$\mathbf{T}$ is an (unsigned) primed tableau of shape $\lambda$ and $\mathbf{Q}$ is a standard shifted tableau of 
the same shape. 

An \defn{(unsigned) primed tableau} of shape $\lambda$ (cf. semistandard $P$-tableau~\cite{Lam.1995} or
semistandard marked shifted tableau~\cite{Cho.2013}) is a signed primed tableau $\mathbf{T}$ of shape $\lambda$ with 
only unprimed elements on the main diagonal. Denote the set of primed tableaux of shape $\lambda$ by 
$\mathcal{PT}(\lambda)$. The weight function  $\wt(\mathbf{T})$ of $\mathbf{T} \in \mathcal{PT}(\lambda)$ is inherited 
from the weight function of signed primed tableaux, that is, it is the vector with $i$-th coordinate equal to the number of 
letters $i'$ and $i$ in $\mathbf{T}$. We can simplify~\eqref{equation.PTpm} as
\begin{equation}
\label{equation.PT}
	F^C_{w}(\mathbf{x}) = \sum_{\lambda} 2^{\ell(\lambda)} \big|\mathcal{UT}_w (\lambda) \big|  
	\sum_{\mathbf{T} \in \mathcal{PT}(\lambda)} \mathbf{x}^{\wt(\mathbf{T})}.
\end{equation}

\begin{remark} The sum $\sum_{\mathbf{T} \in \mathcal{PT}(\lambda)} \mathbf{x}^{\wt(\mathbf{T})}$ is also known as a 
$P$-Schur function.
\end{remark}

Given a word $b_1 b_2 \ldots b_h$ in the alphabet $X = \{1<2<3<\cdots\}$, we recursively construct a sequence of tableaux 
$(\emptyset, \emptyset) = (\mathbf{T}_0, \mathbf{Q}_0),$ $(\mathbf{T}_1, \mathbf{Q}_1), \ldots, (\mathbf{T}_h, \mathbf{Q}_h) 
= (\mathbf{T}, \mathbf{Q})$, where $\mathbf{T}_s \in \mathcal{PT}(\lambda^{(s)})$ and $\mathbf{Q}_s \in 
\mathcal{ST}(\lambda^{(s)})$. 
To obtain the tableau $\mathbf{T}_{s}$, insert the letter $b_s$ into $\mathbf{T}_{s-1}$ as follows. 
First, insert $b_s$ into the first row of $\mathbf{T}_{s-1}$, bumping out the leftmost element $y$ that is strictly greater 
than $b_i$ in the alphabet $X' = \{1' < 1 < 2' < 2< \cdots \}$.
\begin{enumerate}
	\item If $y$ is not on the main diagonal and $y$ is not primed, then insert it into the next row, bumping out the leftmost 
    	element that is strictly greater than $y$ from that row.
	\item If $y$ is not on the main diagonal and $y$ is primed, then insert it into the next column to the right, bumping out 
	the topmost element that is strictly greater than $y$ from that column.
	\item If $y$ is on the main diagonal, then it must be unprimed. Prime $y$ and insert it into the column on the right, 
	bumping out the topmost element that is strictly greater than $y$ from that column.
\end{enumerate}
If a bumped element exists, treat it as a new $y$ and repeat the steps above -- if the new $y$ is unprimed, row-insert 
it into the row below its original cell, and if the new $y$ is primed, column-insert it into the column to the right of its original cell.

The insertion process terminates either by placing a letter at the end of a row, bumping no new element, or forming a 
new row with the last bumped element.

\begin{example}
Under mixed insertion, 
$$\young(22\threep3,:33) \leftarrow 1 = \young(1\twop\threep3,:2\threep,::3).$$
Let us explain each step in detail. The letter $1$ is inserted into the first row bumping out the $2$ from the main diagonal, 
making it a $2'$, which is then inserted into the second column. The letter $2'$ bumps out $2$, which we insert into the 
second row. Then $3$ from the main diagonal is bumped from the second row, making it a $3'$, which is then inserted 
into third column. The letter $3'$ bumps out the 3 on the second row, which is then inserted as the first element in the 
third row.
\end{example}

The shapes of $\mathbf{T}_{s-1}$ and $\mathbf{T}_s$ differ by one box. Add that box to $\mathbf{Q}_{s-1}$ with a letter 
$s$ in it, to obtain the standard shifted tableau $\mathbf{Q}_s$.

\begin{example}
For a word $332332123$, some of the tableaux in the sequence $(\mathbf{T}_i, \mathbf{Q}_i)$ are
$$\Big(\ \young(2\threep,:3),\young(12,:3) \ \Big), \quad \Big(\ \young(22\threep3,:33),\young(1245,:36) \ \Big), 
\quad \Bigg(\ \young(1\twop2\threep3,:2\threep3,::3),\young(12459,:368,::7) \ \Bigg).$$
\end{example}

\begin{theorem} \cite{Haiman.1989}
The construction above gives a bijection
\begin{equation*}
	\mathrm{HM} \colon \mathcal{B}^h \rightarrow \bigcup_{\lambda\vdash h} \big[ \mathcal{PT}(\lambda) \times 
	\mathcal{ST}(\lambda) \big].
\end{equation*}
\end{theorem}

The bijection $\mathrm{HM}$ is called a \defn{mixed insertion}. If $\mathrm{HM}(\mathbf{b}) = (\mathbf{T},\mathbf{Q})$, 
denote $P_{\mathrm{HM}} (\mathbf{b}) = \mathbf{T}$ and $R_{\mathrm{HM}}(\mathbf{b}) = \mathbf{Q}$.

Just as for the RSK-algorithm, the mixed insertion has the property of preserving the recording tableau within each 
connected component of the crystal $\mathcal{B}^h$.

\begin{theorem}
\label{theorem.main0}
The recording tableau $R_{\mathrm{HM}} (\cdot)$ is constant on each connected component of the crystal 
$\mathcal{B}^h$.
\end{theorem}

Before we provide the proof of Theorem~\ref{theorem.main0}, we need to define one more insertion 
from~\cite{Haiman.1989}, which serves as a dual to the previously discussed mixed insertion.

We use the notion of \defn{generalized permutations}. Similar to a regular permutation in two-line notation, a 
generalized permutation $w$ consists of two lines 
$\binom{a_1 a_2\cdots a_h}{b_1 b_2 \cdots b_h}$, which gives a 
correspondence between $a_s$ and $b_s$, but there can be repeated letters now. We order the pairs 
$(a_s, b_s)$ by making the top line weakly increasing $a_1 \leqslant\cdots \leqslant a_h$, and forcing 
$b_{s} \leqslant b_{s+1}$ whenever $a_s = a_{s+1}$. The inverse of a generalized permutation $w^{-1}$ consists of 
pairs $(b_s, a_s)$, ordered appropriately. Given a word $\mathbf{b} = b_1\ldots b_h$, it can be represented as a generalized 
permutation $w$ by setting the first line of the permutation to be $1\ 2\ \ldots h$ and the second line to be 
$b_1\ b_2\ \ldots b_h$. Since the inverse of the generalized permutation $w$ exists, it also defined $\mathbf{b}^{-1}$.

Now, let $w=\binom{a_1 a_2\cdots a_h}{b_1 b_2 \cdots b_h}$ be a generalized permutation on the alphabet $X$,
where the second line consists of distinct letters. We recursively construct a sequence of tableaux 
$(\emptyset, \emptyset) = (\mathbf{Q}_0, \mathbf{T}_0),$ $(\mathbf{Q}_1, \mathbf{T}_1), \ldots, (\mathbf{Q}_h, \mathbf{T}_h)
 = (\mathbf{Q}, \mathbf{T})$, where $\mathbf{Q}_s \in \mathcal{ST}(\lambda_s)$ and $\mathbf{T}_s 
 \in \mathcal{PT}(\lambda_s)$. To obtain the tableau $\mathbf{Q}_{s}$, insert the letter $b_s$ into $\mathbf{Q}_{s-1}$ 
 as follows:
\begin{itemize}
\item
Insert $b_s$ into the first row of $\mathbf{Q}_{s-1}$, and insert each bumped element into the next row until either 
an element is inserted into an empty cell and the algorithm terminates, or an element $b$ has been bumped from the 
diagonal. In the latter case, insert $b$ into the column to its right and continue bumping by columns, until an empty cell is 
filled.
\item
The shapes of $\mathbf{Q}_{s-1}$ and $\mathbf{Q}_s$ differ by one box. Add that box to $\mathbf{T}_{s-1}$ with a letter 
$a_s$ in it. Prime that letter if a diagonal element has been bumped in the process of inserting $b_s$ into 
$\mathbf{Q}_{s-1}$. 
\end{itemize}

The above insertion process is called a \defn{Worley--Sagan insertion algorithm}. The insertion tableau 
$\mathbf{Q}$ will be denoted by $P_{\mathrm{WS}} (w)$ and the recording tableau $\mathbf{T}$ is denoted by 
$R_{\mathrm{WS}} (w)$.

\begin{theorem} \cite[Theorem 6.10 and Corollary 6.3]{Haiman.1989}
\label{theorem.insertion dual}
Given $\mathbf{b} \in \mathcal{B}^h$, we have $R_{\mathrm{HM}} (\mathbf{b}) = P_{\mathrm{WS}} (\mathbf{b}^{-1})$.
\end{theorem}
Next, we want to find out when the Worley--Sagan insertion tableau is preserved. Fortunately, other results 
from~\cite{Haiman.1989} provide this description.
\begin{theorem} \cite[Corollaries 5.8 and 6.3]{Haiman.1989}
\label{theorem.haiman WS}
If two words with distinct letters $\mathbf{b}$ and $\mathbf{b}'$ are related by a shifted Knuth transformation, 
then $P_{\mathrm{WS}} (\mathbf{b}) = P_{\mathrm{WS}} (\mathbf{b}')$.
\end{theorem}

Here, a \defn{shifted Knuth transformation} is an exchange of consecutive letters in one of the following forms:
\begin{enumerate}
\item Knuth transformations: $cab \leftrightarrow acb$ or $bca \leftrightarrow bac$, where $a<b<c$,
\item Worley--Sagan transformation: $xy \leftrightarrow yx$, where $x$ and $y$ are the first two letters of the word.
\end{enumerate}

We are now ready to prove the theorem.

\begin{proof}[Proof of Theorem~\ref{theorem.main0}]
If $\mathbf{b}$ and $\mathbf{b}'$ are two words in the same connected component of $\mathcal{B}^h$, their RSK-recording
tableaux $R_{\mathrm{RSK}} (\mathbf{b})$ and $R_{\mathrm{RSK}} (\mathbf{b}')$ are the same. Thus, 
$P_{\mathrm{RSK}} (\mathbf{b}^{-1})$ and $P_{\mathrm{RSK}} (\mathbf{b}'^{-1})$ are the same, and the second lines of 
$\mathbf{b}^{-1}$ and $\mathbf{b}'^{-1}$ are related by a sequence of Knuth transformations. This in turn means that 
$P_{\mathrm{WS}} (\mathbf{b}^{-1})$ and $P_{\mathrm{WS}} (\mathbf{b}'^{-1})$ are the same, and 
$R_{\mathrm{HM}} (\mathbf{b}) = R_{\mathrm{HM}} (\mathbf{b}')$ by Theorem~\ref{theorem.haiman WS}.
\end{proof}

Let us fix a recording tableau $\mathbf{Q}_{\lambda} \in \mathcal{ST} (\lambda)$. Define a map 
$\Psi_\lambda \colon \mathcal{PT}(\lambda) \rightarrow \mathcal{B}^{h}$ as $\Psi_\lambda (\mathbf{T}) = 
\mathrm{HM}^{-1} (\mathbf{T}, \mathbf{Q}_\lambda)$. By Theorem~\ref{theorem.main0}, the set  
$\mathrm{Im}(\Psi_{\lambda})$ consists of several connected components of $\mathcal{B}^h$. 
The map $\Psi_{\lambda}$ can thus be taken as a crystal isomorphism, and we can define the crystal operators 
and weight function on $\mathcal{PT}(\lambda)$ as
\begin{equation}
\label{equation.ef}
	e_i(\mathbf{T}) := (\Psi_\lambda^{-1} \circ e_i \circ \Psi_\lambda) (\mathbf{T}), \quad f_i(\mathbf{T}) 
	:= (\Psi_\lambda^{-1} \circ f_i \circ \Psi_\lambda) (\mathbf{T}), \quad \wt(\mathbf{T}) := (\wt \circ \Psi_\lambda) (\mathbf{T}).
\end{equation}

Although it is not clear that the crystal operators constructed above are independent of the choice of $\mathbf{Q}_\lambda$, 
in the next section we will construct explicit crystal operators on the set $\mathcal{PT}(\lambda)$ that satisfy the relations 
above and do not depend on the choice of $\mathbf{Q}_\lambda$.

\begin{example}
For $\mathbf{T} = \young(1\twop2\threep3,:2\threep3,::3)$, choose $\mathbf{Q}_{\lambda} = \young(12345,:678,::9)$.
Then $\Psi_\lambda (\mathbf{T}) = 333332221$ and $e_1 \circ \Psi_\lambda (\mathbf{T}) = 333331221$. Thus,
\begin{equation*}
\ e_1 (\mathbf{T}) = (\Psi_\lambda^{-1} \circ e_1 \circ \Psi_\lambda) (\mathbf{T})  = \young(112\threep3,:2\threep3,::3), \quad f_1(\mathbf{T}) = f_2(\mathbf{T}) = \mathbf{0}.
\end{equation*}
\end{example}

To summarize, we obtain a crystal isomorphism between the crystal $(\mathcal{PT}(\lambda), e_i, f_i, \wt)$, denoted again 
by $\mathcal{PT}(\lambda)$, and a direct sum $\bigoplus_\mu \mathcal{B}_{\mu}^{\oplus h_{\lambda\mu}}$. We will provide a
combinatorial description of the coefficients $h_{\lambda\mu}$ in the next section. This implies the relation on characters 
of the corresponding crystals $\chi_{\mathcal{PT}(\lambda)} = \sum_\mu h_{\lambda\mu} s_\mu$. Thus we can
rewrite~\eqref{equation.PT} one last time
\begin{equation*}
	F^C_{w}(\mathbf{x}) = \sum_{\lambda} 2^{\ell(\lambda)} \big|\mathcal{UT}_w (\lambda) \big|  \sum_{\mu} 
	h_{\lambda\mu} s_\mu = \sum_\mu \Big( \sum_\lambda 2^{\ell(\lambda)} \big|\mathcal{UT}_w (\lambda) 
	\big|\ h_{\lambda\mu} \Big) s_\mu.
\end{equation*}

\section{Explicit crystal operators on shifted primed tableaux}
\label{section.explicit}

We consider the alphabet $X'=\{1' < 1 < 2' < 2 < 3' < \cdots\}$ of primed and unprimed letters. It is useful to think about the 
letter $(i+1)'$ as a number $i + 0.5$. Thus, we say that letters $i$ and $(i+1)'$ differ by half a unit and letters $i$ and 
$(i+1)$ differ by a whole unit.

Given an (unsigned) primed tableau $\mathbf{T}$, we construct the \defn{reading word}  $\mathrm{rw}(\mathbf{T})$ as follows:
\begin{enumerate}
\item List all primed letters in the tableau, column by column, from top to bottom within each column, moving from the 
rightmost column to the left, and with all the primes removed (i.e. all letters are increased by half a unit).
(Call this part of the word the \defn{primed reading word}.)
\item Then list all unprimed elements, row by row, from left to right within each row, moving from the bottommost 
row to the top. (Call this part of the word the \defn{unprimed reading word}.)
\end{enumerate}

To find the letter on which the crystal operator $f_i$ acts, apply the bracketing rule for letters $i$ and $i+1$ within the 
reading word $\mathrm{rw}(\mathbf{T})$. 
If all letters $i$ are bracketed in $\mathrm{rw}(\mathbf{T})$, then $f_i(\mathbf{T}) = \mathbf{0}$. 
Otherwise, the rightmost unbracketed letter $i$ in $\mathrm{rw}(\mathbf{T})$ corresponds to an $i$ or an $i'$ in 
$\mathbf{T}$, which we call \defn{bold unprimed} $i$ or \defn{bold primed} $i$ respectively.

If the bold letter $i$ is unprimed, denote the cell it is located in as $x$.

If the bold letter $i$ is primed, we \textit{conjugate} the tableau $\mathbf{T}$ first.

The \defn{conjugate} of a primed tableau $\mathbf{T}$ is obtained by reflecting the tableau over the main diagonal, 
changing all primed entries $k'$ to $k$ and changing all unprimed elements $k$ to $(k+1)'$ (i.e. increase the entries of all 
boxes by half a unit). The main diagonal is now the North-East boundary of the tableau. Denote the resulting tableau as 
$\mathbf{T}^*$.

Under the transformation $\mathbf{T} \to \mathbf{T}^*$, the bold primed $i$ is transformed into bold unprimed $i$.
Denote the cell it is located in as $x$.

Given any cell $z$ in a shifted primed tableau $\mathbf{T}$ (or conjugated tableau $\mathbf{T}^*$), denote by 
$c(z)$ the entry contained in cell $z$.
Denote by $z_E$ the cell to the right of $z$, $z_W$ the cell to its left, $z_S$ the cell below, and $z_N$ the cell above.
Denote by $z^*$ the corresponding conjugated cell in $\mathbf{T}^*$ (or in $\mathbf{T}$).
Now, consider the box $x_E$ (in $\mathbf{T}$ or in $\mathbf{T}^*$) and notice that $c(x_E) \geqslant (i+1)'$.\\

\noindent
\textbf{Crystal operator $f_i$ on primed tableaux:}

\begin{enumerate}
	\item If $c(x_E) =  (i+1)'$, the box $x$ must lie outside of the main diagonal and the box immediately below $x_E$ cannot
	contain $(i+1)'$. Change $c(x)$ to $(i+1)'$ and
	change $c(x_E)$ to $(i+1)$ (i.e. increase the entry in cell $x$ and $x_E$ by half a unit). 
	\item If $c(x_E) \neq (i+1)'$ or $x_E$ is empty, then there is a 
	maximal connected ribbon (expanding in South and West directions) with the following properties:
	\begin{enumerate}
		\item The North-Eastern most box of the ribbon (the tail of the ribbon) is $x$.
		\item The entries of all boxes within a ribbon besides the tail are either $(i+1)'$ or $(i+1)$.
	\end{enumerate}
	Denote the South-Western most box of the ribbon (the head) as $x_H$. 
	\begin{enumerate}
		\item If $x_H = x$, change $c(x)$ to $(i+1)$ (i.e. increase the 
		entry in cell $x$ by a whole unit).
		\item If $x_H \neq x$ and $x_H$ is on the main diagonal (in case of a tableau $\mathbf{T}$), change $c(x)$ 
		to $(i+1)'$ (i.e. increase the entry in cell $x$ by half a unit).
		\item Otherwise, $c(x_H)$ must be $(i+1)'$ due to the bracketing rule. We change $c(x)$ to $(i+1)'$ 
		and change $c(x_H)$ to $(i+1)$ (i.e. increase the entry in cell $x$ and $x_H$ by half a unit).
	\end{enumerate}
\end{enumerate}

In the case when the bold $i$ in $\mathbf{T}$ is unprimed, we apply the above crystal operator rules to $\mathbf{T}$
to find $f_i(\mathbf{T})$

\begin{example}
We apply operator $f_2$ on the following tableaux. The bold letter is marked if it exists:
\begin{enumerate}
\item $\mathbf{T} = \young(1\twop2\threep,:2\threep3)\ $, $\mathrm{rw}(\mathbf{T}) = 3322312$, thus 
$f_2(\mathbf{T}) = \mathbf{0}$;\\ 

\item $\mathbf{T} = \young(1\twop\boldtwo\threep,:2\threep4)\ $, $\mathrm{rw}(\mathbf{T}) = 3322412$, 
thus $f_2(\mathbf{T}) = \young(1\twop\threep3,:2\threep4)$ by Case (1). \\

\item $\mathbf{T} = \young(112\boldtwo,:3\fourp4)\ $, $\mathrm{rw}(\mathbf{T}) = 4341122$,
thus $f_2(\mathbf{T}) = \young(1123,:3\fourp4)$ by Case (2a).\\

\item $\mathbf{T} = \young(11\twop\boldtwo3,:22\threep,::33)$, $\mathrm{rw}(\mathbf{T}) = 3233221123$, 
thus $f_2(\mathbf{T}) = \young(11\twop\threep3,:22\threep,::33)$ by Case~(2b).\\

\item $\mathbf{T} = \young(111\boldtwo3,:22\threep,::3\fourp)$, $\mathrm{rw}(\mathbf{T}) = 3432211123$, 
thus $f_2(\mathbf{T}) = \young(111\threep3,:223,::3\fourp)$ by Case~(2c).
\end{enumerate}
\end{example}

In the case when the bold $i$ is primed in $\mathbf{T}$, we first conjugate $\mathbf{T}$ and then apply the above crystal 
operator rules on $\mathbf{T}^*$, before reversing the conjugation. 
Note that Case~(2b) is impossible for $\mathbf{T}^*$, since the main diagonal is now on the North-East.

\begin{example}
	\begin{equation*}
		\text{Let} \ 
		\mathbf{T} = \young(1\boldtwop23,:3\fourp,::4)\ ,
		\quad
		\text{then} \ 
		\mathbf{T}^* = \young(\twop,\boldtwo\fourp,\threep4\fivep,\fourp)
		\quad
		\text{and} \
		f_2 (\mathbf{T}) = \young(12\threep3,:3\fourp,::4)\ .
	\end{equation*}
\end{example}

\begin{theorem}
\label{theorem.main2}
	 For any $\mathbf{b} \in \mathcal{B}^h$ with $P_{\mathrm{HM}}(\mathbf{b}) = \mathbf{T}$ and 
	 $f_i(\mathbf{b})\neq \mathbf{0}$, the operator $f_i$ defined on  above satisfies
	\begin{equation*}
	P_{\mathrm{HM}}(f_i(\mathbf{b})) = f_i(\mathbf{T}).
	\end{equation*}
	Also, $f_i(\mathbf{b}) = \mathbf{0}$ if and only if $f_i(\mathbf{T})=\mathbf{0}$.
\end{theorem}

The proof of Theorem~\ref{theorem.main2} is quite technical and is relegated to Appendix~\ref{section.proof main2}.  
It implies that the explicit operators $f_i$ in this section are indeed equal to those defined in~\eqref{equation.ef}
and that they are independent of the choice of $\mathbf{Q}_\lambda$. 
We also immediately obtain:

\begin{proof}[Second proof of Theorem~\ref{theorem.main0}]
Given a word $\mathbf{b}=b_1\ldots b_h$, let $\mathbf{b}'= f_i(\mathbf{b}) = b'_1 \ldots b'_h$, so that $b_m \neq b'_m$ 
for some $m$ and $b_i = b'_i$ for any $i \neq m$. We show that $Q_{\mathrm{HM}} (\mathbf{b}) = 
Q_{\mathrm{HM}} (\mathbf{b}')$.

Denote $\mathbf{b}^{(s)} = b_1\ldots b_s$ and similarly $\mathbf{b}'^{(s)} = b'_1\ldots b'_s$.
Due to the construction of the recording tableau $Q_{\mathrm{HM}}$, it suffices to show that 
$P_{\mathrm{HM}}(\mathbf{b}^{(s)})$ and $P_{\mathrm{HM}}(\mathbf{b}'^{(s)})$ have the same shape for any 
$1 \leqslant s \leqslant h$.

If $s < m$, this is immediate. If $s \geqslant m$, note that $\mathbf{b}'^{(s)}=f_i(\mathbf{b}^{(s)})$.
Using Theorem~\ref{theorem.main2}, one can see that $P_{\mathrm{HM}} (\mathbf{b}'^{(s)}) 
= P_{\mathrm{HM}}(f_i(\mathbf{b}^{(s)})) = f_i(P_{\mathrm{HM}}(\mathbf{b}^{(s)}))$ has the same shape 
as $P_{\mathrm{HM}}(\mathbf{b}^{(s)})$.
\end{proof} 

The next step is to describe the raising operators $e_i (\mathbf{T})$.
Consider the reading word $\mathrm{rw}(\mathbf{T})$ and apply the bracketing rule on the letters $i$ and $i+1$.
If all letters $i+1$ are bracketed in $\mathrm{rw}(\mathbf{T})$, then $e_i(\mathbf{T}) = \mathbf{0}$. 
Otherwise, the leftmost unbracketed letter $i+1$ in $\mathrm{rw}(\mathbf{T})$ corresponds to an $i+1$ or an $(i+1)'$ in 
$\mathbf{T}$, which we will call bold unprimed $i+1$ or bold primed $i+1$, respectively.
If the bold $i+1$ is unprimed, denote the cell it is located in by $y$.
If the bold $i+1$ is primed, conjugate $\mathbf{T}$ and denote the cell with the bold $i+1$ in $\mathbf{T}^*$ by $y$.\\

\noindent
\textbf{Crystal operator $e_i$ on primed tableaux:}
\begin{enumerate}
	\item If $c(y_W) =  (i+1)'$, then change $c(y)$ to $(i+1)'$ and
	change $c(y_W)$ to $i$ (i.e. decrease the entry in cell $y$ and $y_W$ by half a unit). 
	\item If $c(y_W) < (i+1)'$ or $y_W$ is empty, then there is a 
	maximal connected ribbon (expanding in North and East directions) with the following properties:
	\begin{enumerate}
		\item The South-Western most box of the ribbon (the head of the ribbon) is $y$.
		\item The entry in all boxes within a ribbon besides the tail is either $i$ or $(i+1)'$.
	\end{enumerate}
	Denote the North-Eastern most box of the ribbon (the tail) as $y_T$. 
	\begin{enumerate}
		\item If $y_T = y$, change $c(y)$ to $i$ (i.e. decrease the 
		entry in cell $y$ by a whole unit).
		\item If $y_T \neq y$ and $y_T$ is on the main diagonal (in case of a conjugate tableau $\mathbf{T}^*$), 
		then change $c(y)$ to 
		$(i+1)'$ (i.e. decrease the entry in cell $y$ by half a unit). 
		\item If $y_T \neq y$ and $y_T$ is not on the diagonal, the entry of cell $y_T$ must be $(i+1)'$ 
		and we change $c(y)$ to $(i+1)'$ and change $c(y_T)$ to $i$ (i.e. decrease the entry of cell $y$ 
		and $y_T$ by half a unit).
	\end{enumerate}
\end{enumerate}
When the bold $i+1$ is unprimed, $e_i(\mathbf{T})$ is obtained by applying the rules above to $\mathbf{T}$. 
When the bold $i+1$ is primed, we first conjugate $\mathbf{T}$, then apply the raising crystal operator rules on 
$\mathbf{T}^*$, and then reverse the conjugation. 

\begin{proposition}
	\begin{equation*}
		e_i (\mathbf{b}) = \mathbf{0} \quad \text{if and only if} \quad e_i (\mathbf{T}) = \mathbf{0}.
	\end{equation*}
\end{proposition}

\begin{proof}
According to Lemma~\ref{lemma.main}, the number of unbracketed letters $i$ in $\mathbf{b}$ is equal to the number of 
unbracketed letters $i$ in $\mathrm{rw}(\mathbf{T})$. Since the total number of both letters $i$ and $j=i+1$ is the same 
in $\mathbf{b}$ and in $\mathrm{rw}(\mathbf{T})$, that also means that the number of unbracketed letters $j$ in 
$\mathbf{b}$ is equal to the number of unbracketed letters $j$ in $\mathrm{rw}(\mathbf{T})$.
Thus, there are no unbracketed letters $j$ in $\mathbf{b}$ if and only if there are no unbracketed letters $j$ in $\mathbf{T}$.
\end{proof}

\begin{theorem}
\label{theorem.main3}
	Given a primed tableau $\mathbf{T}$ with $f_i(\mathbf{T}) \neq \mathbf{0}$, for the operators $e_i$ defined 
	above we have the following relation:
	\begin{equation*}
	e_i(f_i(\mathbf{T})) = \mathbf{T}.
	\end{equation*}
\end{theorem}

The proof of Theorem~\ref{theorem.main3} is relegated to Appendix~\ref{section.proof main3}.

\begin{corollary}
\label{theorem.main4}
	 For any $\mathbf{b} \in \mathcal{B}^h$ with $\mathrm{HM}(\mathbf{b}) = (\mathbf{T},\mathbf{Q})$, the operator 
	 $e_i$ defined above satisfies
	\begin{equation*}
	\mathrm{HM}(e_i(\mathbf{b})) = (e_i(\mathbf{T}), \mathbf{Q}),
	\end{equation*}
	given the left-hand side is well-defined.
\end{corollary}

The consequence of Theorem~\ref{theorem.main2}, as discussed in Section~\ref{section.implicit}, is a crystal 
isomorphism $\Psi_\lambda \colon \mathcal{PT}(\lambda) \rightarrow \bigoplus \mathcal{B}_{\mu}^{\oplus h_{\lambda\mu}}$. 
Now, to determine the nonnegative integer coefficients $h_{\lambda\mu}$, it is enough to count the highest weight elements
in $\mathcal{PT}(\lambda)$ of given weight $\mu$. 

\begin{proposition}
\label{proposition.highest}
	A primed tableau $\mathbf{T} \in \mathcal{PT}(\lambda)$ is a highest weight element if and only if its reading word 	
	$\mathrm{rw}(\mathbf{T})$ is a Yamanouchi word. That is, for any suffix of $\mathrm{rw}(\mathbf{T})$, its weight 
	is a partition.
\end{proposition}

Thus we define $h_{\lambda\mu}$ to be the number of primed tableaux $\mathbf{T}$ of shifted shape $\mathcal{S}(\lambda)$ 
and weight $\mu$ such that $\mathrm{rw}(\mathbf{T})$ is Yamanouchi.

\begin{example}
	Let $\lambda = (5,3,2)$ and $\mu = (4,3,2,1)$. There are three primed tableaux of shifted shape 
	$\mathcal{S}((5,3,2))$ and weight $(4,3,2,1)$ with a Yamanouchi reading word, namely
	\begin{equation*}
		\young(1111\twop,:22\threep,::3\fourp)
		\ , \quad
		\young(1111\threep,:222,::3\fourp)
		\quad \text{and} \quad
		\young(1111\fourp,:222,::33)\ .
	\end{equation*}
	Therefore $h_{(5,3,2)(4,3,2,1)} = 3$.
\end{example}

We summarize our results for the type $C$ Stanley symmetric functions as follows.
\begin{corollary}
\label{corollary.main2}
	The expansion of $F^C_w(\mathbf{x})$ in terms of Schur symmetric functions is
	\begin{equation}
	\label{equation.FC}
	F^C_w(\mathbf{x}) = \sum_\lambda g_{w\lambda} s_\lambda (\mathbf{x}), \quad
	\text{where} \quad
	g_{w\lambda} = \sum_\mu 2^{\ell(\mu)} \big|\mathcal{UT}_w (\mu) \big| \ h_{\mu\lambda}\ .
	\end{equation}
\end{corollary}

Replacing $\ell(\mu)$ by $\ell(\mu)-o(w)$ gives the Schur expansion of $F^B_w(\mathbf{x})$.  Note that since any
row of a unimodal tableau contains at most one zero, $\ell(\mu)-o(w)$ is nonnegative. Thus the given expansion 
makes sense combinatorially.

\begin{example}\label{exa}
Consider the word $w=0101=1010$. There is only one unimodal tableau corresponding to $w$, namely 
$\mathbf{P} = \young(101,:0)$, which belongs to $\mathcal{UT}_{0101} (3,1)$. Thus, $g_{w\lambda} = 4h_{(3,1)\lambda}$. 
There are only three possible highest weight primed tableaux of shape $(3,1)$, namely $\young(111,:2),\ \young(11\twop,:2)$
and $\young(11\threep,:2)$, which implies that $h_{(3,1)(3,1)}= h_{(3,1)(2,2)} = h_{(3,1)(2,1,1)} = 1$ and 
$h_{(3,1)\lambda} = 0$ for other weights $\lambda$. The expansion of $F^C_{0101}(\mathbf{x})$ is thus
\begin{equation*}
	F^C_{0101} = 4s_{(3,1)} + 4s_{(2,2)} + 4s_{(2,1,1)}.
\end{equation*}
\end{example}

\begin{remark}
\label{remark.doubling}
In~\cite[Section 5]{Haiman.1989}, Haiman showed  that shifted mixed insertion can be understood in terms of 
nonshifted mixed insertion operators that produce a symmetric tableau, which can subsequently be cut along the
diagonal. More precisely, starting with a word $\mathbf{b}$, consider its doubling $\mathrm{double}(\mathbf{b})$ by 
replacing each letter $\ell$ by $-\ell \;\ell$. By~\cite[Proposition 6.8]{Haiman.1989} the mixed insertion of 
$\mathrm{double}(\mathbf{b})$ is the symmetrized version of $P_{\mathrm{HM}}(\mathbf{b})$. This symmetrized version 
can also be obtained by first applying usual insertion to obtain $P(\mathrm{double}(\mathbf{b}))$ and then applying 
conversion~\cite[Proposition 14]{SW.2001}. Since both doubling (where the operators are also replaced by their doubled 
versions) and regular insertion commute with crystal operators, it follows that our crystal operators $f_i$ on primed tableaux
can be described as follows:  To apply $f_i$ to $\mathbf{T}$, first form the symmetrization of $\mathbf{T}$ and then apply 
inverse conversion (changing primed entries to negatives).  Next apply the doubled operator $f_if_{-i}$,  and then 
convert ``forwards"  (negatives to primes).  This produces a symmetric tableau, which can then be cut along the diagonal 
to obtain $f_i(\mathbf{T})$.
\end{remark}

\section{Semistandard unimodal tableaux}
\label{section.semistandard}

Many of the results of this paper have counterparts which involve the notion of semi\-standard unimodal tableaux in 
place of primed tableaux.  We give a brief overview of these results, mostly without proof.  

First, let us define semistandard unimodal tableaux. We say that a word $a_1 a_2 \ldots a_h \in \mathcal{B}^h$ 
is \defn{weakly unimodal} if there exists an index $v$, such that 
\[
	a_1 > a_2 > \cdots > a_v \leqslant a_{v+1} \leqslant \cdots \leqslant a_h.
\]  
A \defn{semistandard unimodal tableau} $\mathbf{P}$ of shape $\lambda$ is a filling of $\mathcal{S}(\lambda)$ with 
letters from the alphabet $X$ such that the $i^{th}$ row of $\mathbf{P}$, denoted by $P_i$, is weakly unimodal, and such that 
$P_i$ is the longest weakly unimodal subword in the concatenated word $P_{i+1} P_i$.  Denote the set of 
semistandard unimodal tableaux of shape $\lambda$ by $\mathcal{SUT}(\lambda)$.

Let $\mathbf{a}=a_1\ldots a_h \in \mathcal{B}^h$. The alphabet $X$ imposes a partial order on the entries of $\mathbf{a}$.
We can extend this to a total order by declaring that if $a_i=a_j$ as elements of $X$, and $i<j$, then as entries of 
$\mathbf{a}$, $a_i<a_j$. For each entry $a_i$, denote its numerical position in the total ordering on the entries of 
$\mathbf{a}$ by $n_i$ and define the \defn{standardization} of $\mathbf{a}$ to be the word with superscripts, 
$n_1^{a_1} \ldots n_h^{a_h}$. Since its entries are distinct, $n_1 \ldots n_h$ can be considered as a reduced word. 
Let $(\mathbf{R},\mathbf{S})$ be the Kra\'skiewicz insertion and recording tableaux of $n_1 \ldots n_h$, and let 
$\mathbf{R}^*$ be the tableau obtained from $\mathbf{R}$ by replacing each $n_i$ by $a_i$. One checks that setting 
$\mathrm{SK}(\mathbf{a})=(\mathbf{R}^*,\mathbf{S})$ defines a map,
\[
	\mathrm{SK} \colon  \mathcal{B}=\bigoplus_{h \in \mathbb{N}} \mathcal{B}^h \rightarrow \bigcup_{\lambda} 
	\big[\mathcal{SUT} (\lambda) \times \mathcal{ST} (\lambda)\big].
\]
In fact, this map is a bijection \cite{Serrano.2010,Lam.1995}. It follows that the composition 
$\mathrm{SK} \circ \mathrm{HM}^{-1}$ gives a bijection
\[
	\bigcup_{\lambda} \big[\mathcal{PT} (\lambda) \times \mathcal{ST} (\lambda)\big] \rightarrow \bigcup_{\lambda} 
	\big[\mathcal{SUT} (\lambda) \times \mathcal{ST} (\lambda)\big].
\]
The following remarkable fact, which appears as \cite[Proposition 2.23]{Serrano.2010}, can be deduced
from \cite[Theorem 3.32]{Lam.1995}, which itself utilizes results of \cite{Haiman.1989}.

\begin{theorem}
\label{theorem.same}
	For any word $\mathbf{a}\in \mathcal{B}^h$, $Q_{\mathrm{SK}}(\mathbf{a}) = Q_{\mathrm{HM}}(\mathbf{a})$.
\end{theorem}

This allows us to define a bijective map $\Phi_{\mathbf{Q}} \colon \mathcal{PT} (\lambda) \rightarrow \mathcal{SUT} (\lambda)$ 
as follows. Choose a standard shifted tableau $\mathbf{Q}$ of shape $\lambda$. Then, given a primed tableau 
$\mathbf{P}$ of shape $\lambda$ set $(\mathbf{R}, \mathbf{Q}) = \mathrm{SK}(\mathrm{HM}^{-1}(\mathbf{P},\mathbf{Q}))$,
and let $\Phi_{\mathbf{Q}}(\mathbf{P})=\mathbf{R}$.

For any filling of a shifted shape $\lambda$ with letters from $X$, associating this filling to its reading word (the element of 
$\mathcal{B}^{|\lambda|}$ obtained by reading rows left to right, bottom to top) induces crystal operators on the set of all 
fillings of this shape. In particular, we can apply these induced operators to any element of $\mathcal{SUT} (\lambda)$ 
(although, a priori, it is not clear that the image will remain in $\mathcal{SUT} (\lambda)$). We now summarize our main 
results for SK insertion and its relation to this induced crystal structure.

\begin{theorem}
\label{theorem.main2'}
	 For any $\mathbf{b} \in \mathcal{B}^h$ with $\mathrm{SK}(\mathbf{b}) = (\mathbf{T},\mathbf{Q})$ and 
	 $f_i(\mathbf{b})\neq \mathbf{0}$, the induced operator $f_i$ described above satisfies
	\begin{equation*}
		\mathrm{SK}(f_i(\mathbf{b})) = (f_i(\mathbf{T}), \mathbf{Q}).
	\end{equation*}
	Also, $f_i(\mathbf{b}) = \mathbf{0}$ if and only if $f_i(\mathbf{T})=\mathbf{0}$.
\end{theorem}

\begin{corollary}
	$\mathcal{SUT} (\lambda)$ is closed under the induced crystal operators described above.
\end{corollary}

Replacing $\mathrm{HM}$ by $\mathrm{SK}$ in the second proof of Theorem~\ref{theorem.main0}, or by combining 
Theorem~\ref{theorem.main0}  with Theorem~\ref{theorem.same} yields:

\begin{theorem}
\label{theorem.main0'}
The recording tableau under $\mathrm{SK}$ insertion is constant on each connected component of the crystal 
$\mathcal{B}^h$.
\end{theorem}

The upshot of all this is the following theorem.

\begin{theorem}
\label{theorem.upshot}
With respect to the crystal operators we have defined on primed tableaux and the induced operators on semistandard 
unimodal tableaux described above, the map $\Phi_Q$ is a crystal isomorphism.
\end{theorem}

\begin{proof}
This says no more than that $\Phi_Q$ is a bijection (which we have established) and that it commutes with the crystal 
operations on primed tableaux and semistandard unimodal tableaux. But this is simply combining 
Theorem~\ref{theorem.main0} with Theorem \ref{theorem.main0'}.
\end{proof}

Theorem~\ref{theorem.upshot} immediately gives us another combinatorial interpretation of the coefficients
$g_{w \lambda}$. Let $k_{\mu \lambda}$ be the number of semistandard unimodal tableaux of shape $\mu$ and 
weight $\lambda$, whose reading words are Yamanouchi (that is, tableaux that are the highest weight elements of 
$\mathcal{SUT}(\mu)$).

\begin{corollary}
\label{corollary.main2'}
	The expansion of $F^C_w(\mathbf{x})$ in terms of Schur symmetric functions is
	\begin{equation*}
	F^C_w(\mathbf{x}) = \sum_\lambda g_{w\lambda} s_\lambda (\mathbf{x}), \quad
	\text{where} \quad
	g_{w\lambda} = \sum_\mu 2^{\ell(\mu)} \big|\mathcal{UT}_w (\mu) \big| \ k_{\mu\lambda}\ .
	\end{equation*}
\end{corollary}

Again, replacing $\ell(\mu)$ by $\ell(\mu)-o(w)$ gives the Schur expansion of $F^B_w(\mathbf{x})$.

\begin{example}
According to Example~\ref{exa}, we should find three highest weight semistandard unimodal tableaux of shape $(3,1)$, 
one for each of the weights $(3,1)$, $(2,2)$, and $(2,1,1)$.  These are  $\young(211,:1),\ \young(211,:2)$ and 
$\young(321,:1)$.
\end{example}

\section{Outlook}

There are several other generalizations of the results in this paper that one could pursue. First of all, 
it would be interesting to consider affine Stanley symmetric functions of type $B$ or $C$. As in affine type $A$,
this would involve a generalization of crystal bases as the expansion is no longer in terms of Schur functions.
Another possible extension is to consider $K$-theoretic analogues of Stanley symmetric functions, such
as the (dual) stable Grothendieck polynomials. In type $A$, a crystal theoretic analysis of dual stable
Grothendieck polynomials was carried out in~\cite{galashin.2015}. Type $D$ should also be considered
from this point of view. Finally, the definition of the reading word $\mathrm{rw}$ of Section~\ref{section.explicit}
and the characterization of highest weight elements in Proposition~\ref{proposition.highest} is very similar to
the reading words in~\cite[Section 3.2]{Liu.2017} in the analysis of Kronecker coefficients.

\appendix
\section{Proof of Theorem~\ref{theorem.main2}}
\label{section.proof main2}

In this appendix, we provide the proof of Theorem~\ref{theorem.main2}.

\subsection{Preliminaries}

We use the fact from \cite{Haiman.1989} that taking only elements smaller or equal to $i+1$ from the word $\mathbf{b}$ and 
applying the mixed insertion corresponds to taking only the part of the tableau $\mathbf{T}$ with elements $\leqslant i+1$.
Thus, it is enough to prove the theorem for a ``truncated'' word $\mathbf{b}$ without any letters greater than $i+1$.
To shorten the notation, we set $j= i+1$ in this appendix. We sometimes also restrict to just the letters $i$ and $j$
in a word $w$. We call this the \defn{$\{i,j\}$-subword} of $w$.

First, in Lemma~\ref{lemma.main} we justify the notion of the reading word $\mathrm{rw}(\textbf{T})$ and provide the 
reason to use a bracketing rule on it. After that, in Section~\ref{section.main.proof} we prove that the action of the 
crystal operator $f_i$ on $\mathbf{b}$ corresponds to the action of $f_i$ on $\mathbf{T}$ after the insertion.

Given a word $\mathbf{b}$, we apply the crystal bracketing rule for its $\{i,j\}$-subword and globally declare the 
rightmost unbracketed $i$ in $\mathbf{b}$ (i.e. the letter the crystal operator $f_i$ acts on) to be a bold $i$.
Insert the letters of $\mathbf{b}$ via Haiman insertion to obtain the insertion tableau $\mathbf{T}$. During this process, 
we keep track of the position of the bold $i$ in the tableau via the following rules. When the bold $i$ from $\mathbf{b}$ is 
inserted into $\mathbf{T}$, it is inserted as the rightmost $i$ in the first row of $\mathbf{T}$ since by definition it is 
unbracketed in $\mathbf{b}$ and hence cannot bump a letter $j$. From this point on, the tableau $\mathbf{T}$ has a 
\defn{special} letter $i$ and we track its position:

\begin{enumerate}
\item If the special $i$ is unprimed, it is always the rightmost $i$ in its row. When a letter $i$ is bumped from this row, 
only one of the non-special letters $i$ can be bumped, unless the special $i$ is the only $i$ in the row. When the non-diagonal 
special $i$ is bumped from its row to the next row, it will be inserted as the rightmost $i$ in the next row.
\item When the diagonal special $i$ is bumped from its row to the column to its right, it is inserted as the bottommost $i'$ 
in the next column.
\item If the special $i$ is primed, it is always the bottommost $i'$ in its column. When a letter $i'$ is bumped from this
column, only one of the non-special letters $i'$ can be bumped, unless the special $i'$ is the only $i'$ in the column. 
When the primed special $i$ is bumped from its column to the next column, it is inserted as the bottommost $i'$ in the 
next column.
\item When $i$ is inserted into a row with the special unprimed $i$, the rightmost $i$ becomes special.
\item When $i'$ is inserted into a column with the special primed $i$, the bottommost primed $i$ becomes special.
\end{enumerate}

\begin{lemma}
\label{lemma.main}
Using the rules above, after the insertion process of $\mathbf{b}$, the special $i$ in $\mathbf{T}$ is the same as 
the rightmost unbracketed $i$ in the reading word $\mathrm{rw}(\mathbf{T})$ (i.e. the definition of the bold $i$ in
$\mathbf{T}$). Moreover, the number of unbracketed letters $i$ in $\mathbf{b}$ is equal to the number of unbracketed 
letters $i$ in $\mathrm{rw}(\mathbf{T})$.
\end{lemma}

\begin{proof}
First, note that since both the number of letters $i$ and the number of letters $j$ are equal in $\mathbf{b}$ and 
$\mathrm{rw}(\mathbf{T})$, the fact that the number of unbracketed letters $i$ is the same implies that the number of
unbracketed letters $j$ must also be the same. We use induction on $1 \leqslant s \leqslant h$, where the letters 
$b_1 \ldots b_s$ of $\mathbf{b}=b_1 b_2 \ldots b_h$ have been inserted using Haiman mixed insertion with the above 
rules. That is, we check that at each step of the insertion algorithm the statement of our lemma stays true. 

The induction step is as follows: Consider the word $b_1 \ldots b_{s-1}$ with a corresponding 
insertion tableau $\mathbf{T}^{(s-1)}$. 
If the bold $i$ in $\mathbf{b}$ is not in $b_1\ldots b_{s-1}$, then $\mathbf{T}^{(s-1)}$ does not contain a special letter $i$.
Otherwise, by induction hypothesis assume that the bold $i$ in $b_1\ldots b_{s-1}$  by the above rules corresponds to the 
special $i$ in $\mathbf{T}^{(s-1)}$, that is, it is in the position corresponding to the rightmost unbracketed $i$ in the 
reading word $\mathrm{rw}(\mathbf{T}^{(s-1)})$. Then we need to prove that for $b_1 \ldots b_s$, the special $i$ 
in $\mathbf{T}^{(s-1)}$ ends up in the position corresponding to the rightmost unbracketed $i$ in the reading word of
$\mathbf{T}^{(s)} = \mathbf{T}^{(s-1)} \leftsquigarrow b_s$.
We also need to verify that the second part of the lemma remains true for $\mathbf{T}^{(s)}$.

Remember that we are only considering ``truncated'' words $\mathbf{b}$ with all letters $\leqslant j$.

\smallskip

\noindent
\textbf{Case 1.}
Suppose $b_s = j$. 
In this case $j$ is inserted at the end of the first row of $\mathbf{T}^{(s-1)}$, and $\mathrm{rw}(\mathbf{T}^{(s)})$ has 
$j$ attached at the end. Thus, both statements of the lemma are unaffected.

\smallskip

\noindent
\textbf{Case 2.} 
Suppose $b_s = i$ and $b_s$ is unbracketed in $b_1 \ldots b_{s-1} b_s$.
Then there is no special $i$ in tableau $\mathbf{T}^{(s-1)}$, and $b_s$ might be the bold $i$ of the word $\mathbf{b}$.
Also, there are no unbracketed letters $j$ in $b_1 \ldots b_{s-1}$, and thus all $j$ in $\mathrm{rw}(\mathbf{T}^{(s-1)})$ 
are bracketed. Thus, there are no letters $j$ in the first row of $\mathbf{T}^{(s-1)}$, and $i$ is inserted in the first 
row of $\mathbf{T}^{(s-1)}$, possibly bumping the letter $j'$ from column $c$ into an empty column $c+1$ in the process.
Note that if $j'$ is bumped, moving it to column $c+1$ of $\mathbf{T}^{(s)}$ does not change the reading word, 
since column $c$ of $\mathbf{T}^{(s-1)}$ does not contain any primed letters other than $j'$.
The reading word of $\mathbf{T}^{(s)}$ is thus the same as $\mathrm{rw}(\mathbf{T}^{(s-1)})$ except for an additional 
unbracketed $i$ at the end. The number of unbracketed letters $i$ in both $\mathrm{rw}(\mathbf{T}^{(s)})$ and 
$b_1 \ldots b_{s-1} b_s$ is thus increased by one compared to $\mathrm{rw}(\mathbf{T}^{(s-1)})$ and $b_1 \ldots b_{s-1}$.
If $b_s$ is the bold $i$ of the word $\mathbf{b}$, the special $i$ of tableau $\mathbf{T}^{(s)}$ is the rightmost $i$ on the 
first row and corresponds to the rightmost unbracketed $i$ in $\mathrm{rw}(\mathbf{T}^{(s)})$. 

\smallskip

\noindent
\textbf{Case 3.} 
Suppose $b_s = i$ and $b_s$ is bracketed with a $j$ in the word $b_1\ldots b_{s-1}$. 
In this case, according to the induction hypothesis, $\mathrm{rw}(\mathbf{T}^{(s-1)})$ has an unbracketed $j$. There are 
two options. 

\smallskip

\noindent
\textbf{Case 3.1.} 
If the first row of $\mathbf{T}^{(s-1)}$ does not contain $j$, $b_s$ is inserted at the end of the first row of $\mathbf{T}^{(s-1)}$, 
possibly bumping $j'$ in the process. 
Regardless, $\mathrm{rw}(\mathbf{T}^{(s)})$ does not change except for attaching an $i$ at the end (see Case 2). 
This $i$ is bracketed with one unbracketed $j$ in $\mathrm{rw}(\mathbf{T}^{(s)})$. The special $i$ (if there was one 
in $\mathbf{T}^{(s-1)}$) does not change its position and the statement of the lemma remains true.

\smallskip

\noindent
\textbf{Case 3.2.} 
If the first row of $\mathbf{T}^{(s-1)}$ does contain a $j$, inserting $b_s$ into $\mathbf{T}^{(s-1)}$ bumps $j$ 
(possibly bumping $j'$ beforehand) into the second row, where $j$ is inserted at the end of the row. 
So, if the first row contains $n \geqslant 0$ elements $i$ and $m \geqslant 1$ elements $j$, the reading 
word $\mathrm{rw}(\mathbf{T}^{(s-1)})$ ends with $\ldots i^n j^m$, and $\mathrm{rw}(\mathbf{T}^{(s)})$ ends with 
$\ldots j i^{n+1} j^{m-1}$.  Thus, the number of unbracketed letters $i$ does not change and if there was a special $i$ 
in the first row, it remains there and it still corresponds to the rightmost unbracketed $i$ in $\mathrm{rw}(\mathbf{T}^{(s)})$.

\smallskip

\noindent
\textbf{Case 4.} 
Suppose $b_s < i$. 
Inserting $b_s$ could change both the primed reading word and unprimed reading word of $\mathbf{T}^{(s-1)}$. 
As long as neither $i$ nor $j$ is bumped from the diagonal, we can treat primed and unprimed changes separately. 

\smallskip

\noindent
\textbf{Case 4.1.} 
Suppose neither $i$ nor $j$ is not bumped from the diagonal during the insertion.
This means that there are no transitions of letters $i$ or $j$ between the primed and the unprimed parts of the reading word.
Thus, it is enough to track the bracketing relations in the unprimed reading word; the bracketing relations in the primed 
reading word can be verified the same way via the transposition. After we make sure that the number of unbracketed letters 
$i$ and $j$ changes neither in the primed nor unprimed reading word, it is enough to consider the case when the 
special $i$ is unprimed, since the case when it is primed can again be checked using the transposition.
To avoid going back and forth, we combine these two processes together in each subcase to follow.

\smallskip

\noindent
\textbf{Case 4.1.1.} 
If there are no letters $i$ and $j$ in the bumping sequence, the unprimed $\{i,j\}$-subword of $\mathrm{rw}(\mathbf{T}^{(s)})$ 
is the same as in $\mathrm{rw}(\mathbf{T}^{(s-1)})$. The special $i$ (if there is one) remains in its position, and thus the 
statement of the lemma remains true.

\smallskip

\noindent
\textbf{Case 4.1.2.} 
Now consider the case when there is a $j$ in the bumping sequence, but no $i$. 
Let that $j$ be bumped from the row $r$. Since there is no $i$ bumped, row $r$ does not contain any letters $i$. 
Thus, bumping $j$ from row $r$ to the end of row $r+1$ does not change the $\{i,j\}$-subword
of $\mathrm{rw}(\mathbf{T}^{(s-1)})$, so the statement of the lemma remains true.

\smallskip

\noindent
\textbf{Case 4.1.3.} 
Consider the case when there is an $i$ in the bumping sequence. Let that $i$ be bumped from the row $r$. 

\smallskip

\noindent
\textbf{Case 4.1.3.1.} 
If there is a (non-diagonal) $j$ in row $r+1$, it is bumped into row $r+2$  ($j'$ may have been bumped in the process). 
Note that in this case the $i$ bumped from row $r$ could not have been a special one. 
If there are $n \geqslant 0$ elements $i$ and $m \geqslant 1$ elements $j$ in row $r$, the part of the reading 
word $\mathrm{rw}(\mathbf{T}^{(s-1)})$ with $\ldots i^n j^m i \ldots$ changes to $\ldots j i^{n+1} j^{m-1} \ldots$ in $\mathrm{rw}(\mathbf{T}^{(s)})$. 
The bracketing relations remain the same, and if row $r+1$ contained a special $i$, it would remain there and would 
correspond to the rightmost $i$ in $\mathrm{rw}(\mathbf{T}^{(s)})$.

\smallskip

\noindent
\textbf{Case 4.1.3.2.} 
If there are no letters $j$ in row $r+1$, and $j'$ in row $r+1$ does not bump a $j$, the 
$\{i,j\}$-subword does not change and the statement of the lemma remains true.

\smallskip

\noindent
\textbf{Case 4.1.3.3.} 
Now suppose there are no letters $j$ in row $r+1$ and $j'$ from row $r+1$ bumps a $j$ from another row.
This can only happen if, before the $i$ was bumped, there was only one $i$ in row $r$ of $\mathbf{T}^{(s-1)}$, there is 
a $j'$ immediately below it, and there is a $j$ in the column to the right of $i$ and in row $r' \leqslant r$. 

If $r'=r$, then after the insertion process, $i$ and $j$ are bumped from row 
$r$ to row $r+1$. Since there was only one $i$ in row $r$ and there are no letters $j$ in row $r+1$, the $\{i,j\}$-subword 
of $\mathrm{rw}(\mathbf{T}^{(s-1)})$ does not change and the statement of the lemma remains true.

Otherwise $r' < r$. Then there are no letters $i$ in row $r'$ and by assumption there is no letter $j$ in row $r+1$. 
Thus, moving $i$ to row $r+1$ and moving $j$ to the row $r'+1$ does not change the $\{i,j\}$-subword of 
$\mathrm{rw}(\mathbf{T}^{(s-1)})$ and the statement of the lemma remains true.

\smallskip

\noindent
\textbf{Case 4.2.} 
Suppose $i$ or $j$ (or possibly both) are bumped from the diagonal in the insertion process.

\smallskip

\noindent
\textbf{Case 4.2.1.} 
Consider the case when the insertion sequence ends with $\quad\cdots \rightarrow z \rightarrow j [j']$ with $z<i$ 
and possibly $ \rightarrow j$ right after it. Let the bumped diagonal $j$ be in column $c$. 
Then columns $1,2, \ldots, c$ of $\mathbf{T}^{(s-1)}$ could only contain elements $\leqslant z$, except for the $j$ on the 
diagonal. Thus, the bumping process just moves $j$ from the unprimed reading word to the primed reading word 
without changing the overall order of the $\{i,j\}$-subword.

\smallskip

\noindent
\textbf{Case 4.2.2.} 
Consider the case when the insertion sequence ends with $\quad \cdots \rightarrow i' \rightarrow i \rightarrow j[j']$ and 
possibly $\rightarrow j$. Let the bumped diagonal $j$ be in row (and column) $r$. Note that $r$ must be the last row of 
$\mathbf{T}^{(s-1)}$. Then $i$ has to be bumped from row $r-1$ (and, say, column $c$) and $i'$ also has to be in row $r-1$ 
(moreover, it has to be the only $i'$ in column $c-1$). 
Also, since there are no letters $j'$ in column $c$ (otherwise it would be in row $r$, which is impossible), 
bumping $i'$ to column $c$ does not change the $\{i,j\}$-subword of $\mathrm{rw}(\mathbf{T}^{(s-1)})$. 
Note that after $i'$ moves to column $c$, there are no $i'$ or $j'$ in columns $1,\ldots, r$, and thus priming $j$ and 
moving it to column $r+1$ does not change the $\{i,j\}$-subword.
If the last row $r$ contains $n$ elements $j$, the $\{i,j\}$-subword of $\mathbf{T}^{(s-1)}$ contains $\ldots j^n i \ldots$ 
and after the insertion it becomes $\ldots j i j^{n-1} \ldots$, where  the left $j$ is from the primed subword. 
Thus, the number of bracketed letters $i$ does not change. 
Also, if we moved the special $i$ in the process, it could only have been the bumped $i'$. Its position in the reading 
word is unaffected.

\smallskip

\noindent
\textbf{Case 4.2.3.} 
The case when the insertion sequence does not contain $i'$, does not bump $i$ from the diagonal, but contains $i$ and 
bumps $j$ from the diagonal is analogous to the previous case.

\smallskip

\noindent
\textbf{Case 4.2.4.} 
Suppose both $i$ and $j$ are bumped from the diagonal. 
That could only be the case with diagonal $i$ bumped from row (and column) $r$, 
bumping another letter $i$ from the row $r$ and column $r+1$, 
and bumping $j$ from row (and column) $r+1$ (and possibly bumping $j$ to row $r+2$ at the end). 
Let the number of letters $i'$ in column $r+1$ be $n$ and let the number of letters $j$ in row $r+1$ be $m$. 

\smallskip

\noindent
\textbf{Case 4.2.4.1} 
Let $m\geqslant 2$. Then the $\{i,j\}$-subword of $\mathrm{rw}(\mathbf{T}^{(s-1)})$ contains $\ldots i^n j^m ii \ldots$ 
and after the insertion it becomes $\ldots j i^{n+1} j i j^{m-2} \ldots$. 
The number of unbracketed letters $i$ stays the same. Since $m \geqslant  2$, the special $i$ of $\mathbf{T}^{(s-1)}$ 
could not  have been involved in the bumping procedure. However, the special $i$ might have been the bottommost $i'$ in 
column $r+1$ of $\mathbf{T}^{(s-1)}$, and after the insertion the special $i$ would still be the bottommost $i'$ in column 
$r+1$ and would correspond to the rightmost unbracketed $i$ in $\mathrm{rw}(\mathbf{T}^{(s)})$:
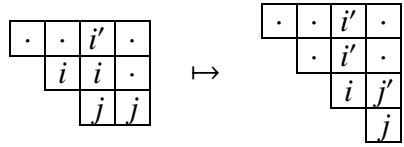
\begin{equation*}
\young(\cdot\cdot\iprime\cdot,:ii\cdot,::jj) \quad \mapsto \quad
\young(\cdot\cdot\iprime\cdot,:\cdot\iprime\cdot,::i\jprime,:::j)
\end{equation*}

\smallskip

\noindent
\textbf{Case 4.2.4.2.} 
Let $m=1$. Then the $\{i,j\}$-subword of $\mathbf{T}^{(s-1)}$ contains $\ldots i^n j ii \ldots$ and after the insertion it 
becomes $\ldots j i^{n+1} i$. The number of unbracketed letters $i$ stays the same.
If the special $i$ was in row $r$ and column $r+1$, then after the insertion it becomes a diagonal one, and it would still 
correspond to the rightmost unbracketed $i$ in $\mathrm{rw}(\mathbf{T}^{(s)})$.

\smallskip

\noindent
\textbf{Case 4.2.5.} 
Suppose only $i$ is bumped from the diagonal (let that $i$ be on row and column $r$). 
Note that there cannot be an $i'$ in column $r$.

\smallskip

\noindent
\textbf{Case 4.2.5.1.} 
Suppose $i$ from the diagonal bumps another $i$ from column $r+1$ and row $r$. 
In that case there are no letters $j$ in row $r+1$. No letters $j$ or $j'$ are affected and thus the 
$\{i,j\}$-subword of $\mathbf{T}^{(s)}$ does not change, and the special $i$ in $\mathbf{T}^{(s)}$ (if there is one) 
still corresponds to the rightmost unbracketed $i$ in $\mathrm{rw}(\mathbf{T}^{(s)})$.

\smallskip

\noindent
\textbf{Case 4.2.5.2.} 
Suppose $i$ from the diagonal bumps $j'$ from column $r+1$ and row $r$. 
Note that $j'$ must be the only $j'$ in column $r+1$. 
Suppose also that there is one $j$ in row $r+1$. 
Denote the number of letters $i'$ in column $r+1$ of $\mathbf{T}^{(s-1)}$ by $n$. 
If there is a $j$ in row $r+1$ of $\mathbf{T}^{(s-1)}$, then the $\{i,j\}$-subword of $\mathbf{T}^{(s-1)}$ contains 
$\ldots i^n jji \ldots$ and after the insertion it becomes $\ldots ji^{n+1}j \ldots$.
If there is no $j$ in row $r+1$ of $\mathbf{T}^{(s-1)}$, then the $\{i,j\}$-subword of $\mathbf{T}^{(s-1)}$ contains 
$\ldots i^n ji \ldots$ and after the insertion it becomes $\ldots ji^{n+1} \ldots$.
The number of unbracketed letters $i$ is unaffected. If the special $i$ of $\mathbf{T}^{(s-1)}$ was the bottommost $i'$ in 
column $r+1$ of $\mathbf{T}^{(s-1)}$, after the insertion the special $i$ is still the bottommost $i'$ in column 
$r+1$ and corresponds to the rightmost unbracketed $i$ in $\mathrm{rw}(\mathbf{T}^{(s)})$.
\end{proof}

\begin{corollary}
\label{corollary.f annihilate}
	\begin{equation*}
		f_i (\mathbf{b}) = \mathbf{0} \quad \text{if and only if} \quad f_i (\mathbf{T}) = \mathbf{0}.
	\end{equation*}
\end{corollary}

\subsection{Proof of Theorem~\ref{theorem.main2}}
\label{section.main.proof}
By Lemma~\ref{lemma.main}, the cell $x$ in the definition of the operator $f_i$ corresponds to the bold $i$ in the 
tableau $\mathbf{T}$. Furthermore, we know how the bold $i$ moves during the insertion procedure. 
We assume that the bold $i$ exists in both $\mathbf{b}$ and $\mathbf{T}$, meaning that $f_i(\mathbf{b}) \neq \mathbf{0}$
and $f_i(\mathbf{T}) \neq \mathbf{0}$ by Corollary~\ref{corollary.f annihilate}.
We prove Theorem~\ref{theorem.main2} by induction on the length of the word $\mathbf{b}$.

\smallskip

\noindent
\textbf{Base.} 
Our base is for words $\mathbf{b}$ with the last letter being a bold $i$ (i.e. rightmost unbracketed $i$). 
Let $\mathbf{b} = b_1 \ldots b_{h-1} b_h$ and $f_i(\mathbf{b}) = b_1 \ldots b_{h-1} b'_h$, where $b_h = i$ and $b'_h = j$. 
Denote the mixed insertion tableau of $b_1 \ldots b_{h-1}$ as $\mathbf{T}_0$, the insertion tableau of 
$b_1 \ldots b_{h-1} b_h$ as $\mathbf{T}$, and the insertion tableau of $b_1 \ldots b_{h-1} b'_h$ as $\mathbf{T}'$. 
Note that $\mathbf{T}_0$ does not  have letters $j$ in the first row. If the first row of $\mathbf{T}_0$ ends with $\ldots j'$, 
then the first row of $\mathbf{T}$ ends with $\ldots \mathbf{i} j'$ and the first row of $\mathbf{T}'$ ends with $\ldots j' j$. 
If the first row of $\mathbf{T}_0$ does not contain $j'$, the first row of $\mathbf{T}$ ends with $\ldots \mathbf{i}$ and the first 
row of $\mathbf{T}'$ ends with $\ldots j$, and the cell $x_S$ is empty. 
In both cases $f_i(\mathbf{T}) = \mathbf{T}'$. 

\smallskip

\noindent
\textbf{Induction step.} 
Now, let $\mathbf{b} = b_1 \ldots b_h$ with operator $f_i$ acting on the letter $b_s$ in $\mathbf{b}$ with $s < h$. 
Denote the mixed insertion tableau of $b_1 \ldots b_{h-1}$ as $\mathbf{T}$ and the insertion tableau of 
$f_i(b_1 \ldots b_{h-1})$ as $\mathbf{T}'$. By induction hypothesis, we know that $f_i(\mathbf{T}) = \mathbf{T}'$. 
We want to show that $f_i(\mathbf{T} \leftsquigarrow b_h) = \mathbf{T}' \leftsquigarrow b_h$.
In Cases 1-3 below, we assume that the bold letter $i$ is unprimed. Since almost all results from the case with unprimed 
$i$ are transferrable to the case with primed bold $i$ via the transposition of the tableau $\mathbf{T}$, we just need to cover 
the differences in Case 4.

\smallskip

\noindent
\textbf{Case 1.} 
Suppose $\mathbf{T}$ falls under Case (1) of the rules for $f_i$: the bold $i$ is in the non-diagonal cell $x$ in row $r$ and 
column $c$ and the cell $x_E$ in the same row and column $c+1$ contains the entry $j'$. Consider the insertion path of $b_h$.

\smallskip

\noindent
\textbf{Case 1.1.} 
If the insertion path of $b_h$ in $\mathbf{T}$ contains neither cell $x$ nor cell $x_E$, the insertion path of $b_h$ in 
$\mathbf{T}'$ also does not contain cells $x$ and $x_E$. Thus, $f_i(\mathbf{T} \leftsquigarrow b_h) 
= \mathbf{T}' \leftsquigarrow b_h$.

\smallskip

\noindent
\textbf{Case 1.2.} 
Suppose that during the insertion of $b_h$ into $\mathbf{T}$, the bold $i$ is row-bumped by an unprimed element 
$d < i$ or is column-bumped by a primed element $d' \leqslant i'$. 
This could only happen if the bold $i$ is the unique $i$ in row $r$ of $\mathbf{T}$. 
During the insertion process, the bold $i$ is inserted into row $r+1$. 
Since there are no letters $i$ in row $r$ of $\mathbf{T}'$, inserting $b_h$ into $\mathbf{T}'$ inserts $d$ in cell $x$, 
bumps $j'$ to cell $x_E$, and bumps $j$ into row $r+1$. Thus we are in a situation similar to the induction base. 
It is easy to check that row $r+1$ does not contain any letters $j$ in $\mathbf{T}$.
If it contains $j'$, this $j'$ is bumped back into 
row $r+1$. Similar to the induction base, $f_i(\mathbf{T} \leftsquigarrow b_h) = \mathbf{T}' \leftsquigarrow b_h$.

\smallskip

\noindent
\textbf{Case 1.3.} 
Suppose that during the insertion of $b_h$ into $\mathbf{T}$, an unprimed $i$ is inserted into row $r$. 
Note that in this case, row $r$ in $\mathbf{T}$ must contain a $j$ (or else the $i$ from row $r$ would not be the 
rightmost unbracketed $i$ in $\mathrm{rw}(\mathbf{T})$). Thus inserting $i$ into row $r$ in $\mathbf{T}$ shifts the bold $i$ to 
column $c+1$, shifts $j'$ to column $c+2$ and bumps $j$ to row $r+1$. Inserting  $i$ into row $r$ in $\mathbf{T}'$ shifts 
$j'$ to column $c+1$ with a $j$ to the right of it, and bumps $j$ into row $r+1$. 
Thus $f_i(\mathbf{T} \leftsquigarrow b_h) = \mathbf{T}' \leftsquigarrow b_h$.

\smallskip

\noindent
\textbf{Case 1.4.}
Suppose that during the insertion of $b_h$ into $\mathbf{T}$, the $j'$ in cell $x_E$ is column-bumped by a primed 
element $d'$ and the cell $x$ is unaffected. Note that in order for $\mathbf{T} \leftsquigarrow b_h$ to be a valid 
primed tableau, $i$ must be smaller than $d'$, and thus $d'$ could only be $j'$. On the other hand, $j'$ cannot be inserted 
into column $c+1$ of $\mathbf{T}'$ in order for  $\mathbf{T}' \leftsquigarrow b_h$ to be a valid primed tableau. 
Thus this case is impossible.

\smallskip

\noindent
\textbf{Case 2.} 
Suppose tableau $\mathbf{T}$ falls under Case (2a) of the crystal operator rules for $f_i$. 
This means that for a bold $i$ in cell $x$ (in row $r$ and column $c$) of tableau $\mathbf{T}$, the cell $x_E$ contains
the entry $j$ or is empty and cell $x_S$ is empty. Tableau $\mathbf{T}'$ has all the same elements as $\mathbf{T}$, 
except for a $j$ in the cell $x$. We are interested in the case when inserting $b_h$ into either $\mathbf{T}$ or 
$\mathbf{T}'$ bumps the element from cell $x$.

\smallskip

\noindent
\textbf{Case 2.1.} 
Suppose that the non-diagonal bold $i$ in $\mathbf{T}$ (in row $r$) is row-bumped by an unprimed element 
$d < i$ or column-bumped by a primed element $d' < j'$. Element $d$ (or $d'$) bumps the bold $i$ into row $r+1$ of 
$\mathbf{T}$, while in $\mathbf{T}'$ (since there are no letters $i$ in row $r$ of $\mathbf{T}'$) it bumps $j$ from cell $x$ 
into row $r+1$. Thus we are in the situation of the induction base and $f_i(\mathbf{T} \leftsquigarrow b_h) 
= \mathbf{T}' \leftsquigarrow b_h$.

\smallskip

\noindent
\textbf{Case 2.2.}
Suppose $x$ is a non-diagonal cell in row $r$, and during the insertion of $b_h$ into $\mathbf{T}$, an unprimed $i$ 
is inserted into the row $r$. In this case, row $r$ in $\mathbf{T}$ must contain a letter $j$.
The insertion process shifts the bold $i$ one cell to the right in $\mathbf{T}$ and bumps a $j$ into row $r+1$, while 
in $\mathbf{T}'$ it just bumps $j$ into the row $r+1$. We end up in Case (2a) of the crystal operator rules for $f_i$
with bold $i$ in the cell $x_E$.

\smallskip

\noindent
\textbf{Case 2.3.}
Suppose that during the insertion of $b_h$ into $\mathbf{T}'$, the $j$ in the non-diagonal cell $x$ is column-bumped 
by a $j'$. This means that $j'$ was previously bumped from column $c-1$ and row $\geqslant r$. Thus the cell 
$x_{SW}$ (cell to the left of an empty $x_{S}$) is non-empty. Moreover, right before inserting $j'$ into the column $c$, 
the cell $x_{SW}$ contains an entry $< j'$. Inserting $j'$ into column $c$ of $\mathbf{T}$ just places $j'$ into the empty cell 
$x_S$. Inserting $j'$ into column $c$ of $\mathbf{T}'$ places $j'$ into $x$, and bumps $j$ into the empty cell $x_S$.
Thus, we end up in Case (2c) of the crystal operator rules after the insertion of $b_h$ with $y = x_S$.

\smallskip

\noindent
\textbf{Case 2.4.} 
Suppose that $x$ in $\mathbf{T}$ is a diagonal cell (in row $r$ and column $r$) and that it is row-bumped by an 
element $d<i$. Note that in this case there cannot be any letter $j$ in row $r+1$.
Also, since $d$ is inserted into cell $x$, there cannot be any letters $i'$ in columns $1,\ldots, r$, and thus there 
cannot be any letters $j'$ in column $r+1$ (otherwise the $i$ in cell $x$ would not be bold).
The bumped bold $i$ in tableau $\mathbf{T}$ is inserted as a primed bold $i'$ into the cell $z$ of column $r+1$.

\smallskip

\noindent
\textbf{Case 2.4.1.}
Suppose that there are no letters $i$ in column $r+1$ of $\mathbf{T}$.
In this case, the cell $z$ in $\mathbf{T}$ either contains $j$ (and then that $j$ would be bumped to the next row) or is empty.
Inserting $b_h$ into tableau $\mathbf{T}'$ bumps the diagonal $j$ in cell $x$, which is inserted as a $j'$ into 
cell $z$, possibly bumping $j$ after that.
Thus, $\mathbf{T} \leftsquigarrow b_h$ falls under Case (2a) of the ``primed'' crystal rules with the bold $i'$ in cell 
$z$ (note that there cannot be any $j'$ in cell $(z*)_E$ of the tableau $(\mathbf{T} \leftsquigarrow b_h)*$). 
Since $\mathbf{T} \leftsquigarrow b_h$ and $\mathbf{T}' \leftsquigarrow b_h$ differ only by the cell $z$, 
$f_i(\mathbf{T} \leftsquigarrow b_h) = \mathbf{T}' \leftsquigarrow b_h$.

\smallskip

\noindent
\textbf{Case 2.4.2.}
Suppose that there is a letter $i$ in cell $z$ of column $r+1$ of $\mathbf{T}$.
Note that cell $z$ can only be in rows $1, \ldots, r-1$ and thus $z_{SW}$ contains an element $< i$.
Thus, during the insertion process of $b_h$ into $\mathbf{T}$, diagonal bold $i$ from cell $x$ is inserted as bold $i'$ 
into cell $z$, bumping the $i$ from cell $z$ into cell $z_S$ (possibly bumping $j$ afterwards).
On the other hand, inserting $b_h$ into $\mathbf{T}'$ bumps the diagonal $j$ from cell $x$ into cell $z_S$ 
as a $j'$ (possibly bumping $j$ afterwards).
Thus, $\mathbf{T} \leftsquigarrow b_h$ falls under Case (1) of the ``primed'' crystal rules with the bold $i'$ in cell 
$z$, and so $f_i(\mathbf{T} \leftsquigarrow b_h) = \mathbf{T}' \leftsquigarrow b_h$.

\smallskip

\noindent
\textbf{Case 2.5.} 
Suppose that $x$ is a diagonal cell (in row $r$ and column $r$) and that during the insertion of $b_h$ into $\mathbf{T}$, an 
unprimed $i$ is inserted into row $r$. In this case, the entry in cell $x_E$ has to be $j$ and the diagonal 
cell $x_{ES}$ must be empty. Inserting $i$ into row $r$ of $\mathbf{T}$ bumps a $j$ from cell $x_E$ into 
cell $x_{ES}$. On the other hand, inserting $i$ into row $r$ of $\mathbf{T}'$ bumps a $j$ from the diagonal cell $x$, 
which in turn is inserted as a $j'$ into cell $x_E$, which bumps $j$ from cell $x_E$ into cell $x_{ES}$.
Thus, $\mathbf{T} \leftsquigarrow b_h$ falls under Case (2b) of the crystal rules with bold $i$ in cell $x_E$ and 
$y= x_{ES}$, and so $f_i(\mathbf{T} \leftsquigarrow b_h) = \mathbf{T}' \leftsquigarrow b_h$.

\smallskip

\noindent
\textbf{Case 3.}
Suppose that $\mathbf{T}$ falls under Case (2b) or (2c) of the crystal operator rules.
That means $x_E$ contains the entry $j$ or is empty and $x_S$ contains the entry $j'$ or $j$. There is a chain of 
letters $j'$ and $j$ in $\mathbf{T}$ starting from $x_S$ and ending on a box $y$.
According to the induction hypothesis, $y$ is either on the diagonal and contains the entry $j$ or $y$ is not on the diagonal and 
contains the entry $j'$. The tableau $\mathbf{T}' = f_i (\mathbf{T})$ has $j'$ in cell $x$ and $j$ in cell $y$.
We are interested in the case when inserting $b_h$ into $\mathbf{T}$ affects cell $x$ or affects some element of the chain.
Let $r_x$ and $c_x$ be the row and the column index of cell $x$, and $r_y$, $c_y$ are defined accordingly.
Note that during the insertion process, $j'$ cannot be inserted into columns $c_y,\ldots, c_x$ and $j$ cannot be
inserted into rows $r_x +1,\ldots, r_y$, since otherwise $\mathbf{T} \leftsquigarrow b_h$ would not be a primed tableau.

\smallskip

\noindent
\textbf{Case 3.1.} 
Suppose the bold $i$ in cell $x$ (of row $r_x$ and column $c_x$) of $\mathbf{T}$ is row-bumped by 
an unprimed element $d < i$ or column-bumped by a primed element $d' < i$.
Note that in this case, bold $i$ in row $r_x$ is the only $i$ in this row, so row $r_x+1$ cannot contain 
any letter $j$. Therefore the entry in cell $x_S$ must be $j'$.
In tableau $\mathbf{T}$, the bumped bold $i$ is inserted into cell $x_S$ and $j'$ is bumped from cell $x_S$ 
into column $c_x+1$, reducing the chain of letters $j'$ and $j$ by one. Notice that since $x_E$ either contains a $j$ 
or is empty,  $j'$ cannot be bumped into a position to the right of $x_S$, so Case (1) of the crystal 
rules for $\mathbf{T} \leftsquigarrow b_h$ cannot occur.  As for $\mathbf{T}'$, inserting $d$ into row $r_x$ (or inserting $d'$ into 
column $c_x$) just bumps $j'$ into column $c_x+1$, thus reducing the length of the chain by one in that tableau as well.
Note that in the case when the length of the chain is one (i.e. $y=x_S$), we would end up in Case (2a) of the crystal 
rules after the insertion. Otherwise, we are still in Case (2b) or (2c). 
In both cases, $f_i(\mathbf{T} \leftsquigarrow b_h) = \mathbf{T}' \leftsquigarrow b_h$.

\smallskip

\noindent
\textbf{Case 3.2.}
Suppose a letter $i$ is inserted into the same row as $x$ (in row $r_x$).
In this case, $x_E$ must contain a $j$ (otherwise the bold $i$ would not be in cell $x$).
After inserting $b_h$ into $\mathbf{T}$, the bold $i$ moves to cell $x_E$ (note that there cannot be a $j'$ to the right of 
$x_E$) and $j$ from $x_E$ is bumped to cell $x_{ES}$, thus the chain now starts at $x_{ES}$.
As for  $\mathbf{T}'$, inserting $i$ into the row $r_x$ moves $j'$ from cell $x$ to the cell $x_E$ and moves $j$ 
from cell $x_E$ to cell $x_{ES}$. Thus, $f_i(\mathbf{T} \leftsquigarrow b_h) = \mathbf{T}' \leftsquigarrow b_h$.

\smallskip

\noindent
\textbf{Case 3.3.}
Consider the chain of letters $j$ and $j'$ in $\mathbf{T}$.
Suppose an element of the chain $z \neq x,y$ is row-bumped by an element $d < j$ or is column-bumped 
by an element $d'<j'$. The bumped element $z$ (of row $r_z$ and column $c_z$) must be a ``corner'' element of the chain, 
i.e. in $\mathbf{T}$ the entry in the boxes must be $c(z)=j', \ c(z_E) = j$ and $c(z_S)$ must be either $j$ or $j'$.
Therefore, inserting $b_h$ into $\mathbf{T}$ bumps $j'$ from box $z$ to box $z_E$ and bumps $j$ from
box $z_E$ to box $z_{ES}$, and inserting $b_h$ into $\mathbf{T}'$ has exactly the same effect.
Thus, there is still a chain of letters $j$ and $j'$ from $x_S$ to $y$ in $\mathbf{T}$ and $\mathbf{T}'$, and 
$f_i(\mathbf{T} \leftsquigarrow b_h) = \mathbf{T}' \leftsquigarrow b_h$.

\smallskip

\noindent
\textbf{Case 3.4.}
Suppose $\mathbf{T}$ falls under Case (2c) of the crystal rules (i.e. $y$ is not a diagonal cell) and during the insertion of 
$b_h$ into $\mathbf{T}$, $j'$ in cell $y$ is row-bumped (resp. column-bumped) by an element $d<j'$ (resp. $d'<j'$).
Since $y$ is the end of the chain of letters $j$ and $j'$, $y_S$ must be empty.
Also, since it is bumped, the entry in $y_E$ must be $j$.
Thus, inserting $b_h$ into $\mathbf{T}$ bumps $j'$ from cell $y$ to cell $y_E$ and bumps $j$ from cell 
$y_E$ into row $r_y+1$ and column $\leqslant c_y$. On the other hand, inserting $b_h$ into $\mathbf{T}'$ bumps $j$ from
cell $y$ into row $r_y+1$ and column $\leqslant c_y$. The chain of letters $j$ and $j'$ now ends at $y_E$ and 
$f_i(\mathbf{T} \leftsquigarrow b_h) = \mathbf{T}' \leftsquigarrow b_h$.

\smallskip

\noindent
\textbf{Case 3.5.}
Suppose $\mathbf{T}$ falls under Case (2b) of the crystal rules (i.e. $y$ with entry $j$ is a diagonal cell) and during 
the insertion of $b_h$ into $\mathbf{T}$, $j$ in cell $y$ is row-bumped by an element $d < j$.
In this case, the cell $y_E$ must contain the entry $j$. Thus, inserting $b_h$ into $\mathbf{T}$ bumps $j$ from cell 
$y$ (making it $j'$) to cell $y_E$ and bumps $j$ from cell $y_E$ to the diagonal cell $y_{ES}$. On the other hand,
inserting $b_h$ into $\mathbf{T}'$ has exactly the same effect. The chain of letters $j$ and $j'$ now ends at 
the diagonal cell $y_{ES}$, so $\mathbf{T}\leftsquigarrow b_h$ falls under Case (2b) of the crystal rules and 
$f_i(\mathbf{T} \leftsquigarrow b_h) = \mathbf{T}' \leftsquigarrow b_h$.

\smallskip

\noindent
\textbf{Case 4.} 
Suppose the bold $i$ in tableau $\mathbf{T}$ is a primed $i$.
We use the transposition operation on $\mathbf{T}$, and the resulting tableau $\mathbf{T}^*$ falls under one of the cases 
of the crystal operator rules. When $b_h$ is inserted into $\mathbf{T}$, we can easily translate the insertion process to 
the transposed tableau $\mathbf{T}^*$ so that $[\mathbf{T}^* \leftsquigarrow (b_h+1)'] = [\mathbf{T} \leftsquigarrow b_h]^*$: 
the letter $(b_h+1)'$ is inserted into the first column of $\mathbf{T}^*$, and all other insertion rules stay exactly same, 
with one exception -- when the diagonal element $d'$ is column-bumped from the diagonal cell of 
$\mathbf{T}^*$, the element $d'$ becomes $(d-1)$ and is inserted into the row below.
Notice that the primed reading word of $\mathbf{T}$ becomes an unprimed reading word of $\mathbf{T}^*$.
Thus, the bold $i$ in tableau $\mathbf{T}^*$ corresponds to the rightmost unbracketed $i$ in the \textit{unprimed} 
reading word of $\mathbf{T}^*$. Therefore, everything we have deduced in Cases 1-3 from the fact that bold $i$ is in the 
cell $x$ will remain valid here. Given $f_i(\mathbf{T}^*) = \mathbf{T}'^*$, we want to make sure that 
$f_i(\mathbf{T}^* \leftsquigarrow (b_h+1)') = \mathbf{T}'^* \leftsquigarrow (b_h+1)'$.

The insertion process of $(b_h+1)'$ into $\mathbf{T}^*$ falls under one of the cases above and the proof of 
$f_i(\mathbf{T}^* \leftsquigarrow (b_h+1)') = \mathbf{T}'^* \leftsquigarrow (b_h+1)'$ is exactly the same as the proof in 
those cases. We only need to check the cases in which the diagonal element might be affected differently in the 
insertion process of $(b_h+1)'$ into $\mathbf{T}^*$ compared to the insertion process of $(b_h+1)'$ into $\mathbf{T}'^*$.
Fortunately, this never happens: in Case 1 neither $x$ nor $x_E$ could be diagonal elements; in Cases 2 and 3 $x$ 
cannot be on the diagonal, and if $x_E$ is on diagonal, it must be empty.
Following the proof of those cases, $f_i(\mathbf{T}^* \leftsquigarrow (b_h+1)') = \mathbf{T}'^* \leftsquigarrow (b_h+1)'$.

\section{Proof of Theorem~\ref{theorem.main3}}
\label{section.proof main3}

This appendix provides the proof of Theorem~\ref{theorem.main3}. In this section we set $j=i+1$.
We begin with two preliminary lemmas.

\subsection{Preliminaries}

\begin{lemma}
\label{lemma.chains}
	Consider a shifted tableau $\mathbf{T}$.
	\begin{enumerate}
	\item Suppose tableau $\mathbf{T}$ falls under Case (2c) of the $f_i$ crystal operator rules, that is, there is a chain of 
	letters $j$ and $j'$ starting from the bold $i$ in cell $x$ and ending at $j'$ in cell $x_H$. 
	Then for any cell $z$ of the chain containing $j$, the cell $z_{NW}$ contains $i$.
	\item  Suppose tableau $\mathbf{T}$ falls under Case (2b) of the $f_i$ crystal operator rules, that is, there is a chain of 
	letters $j$ and $j'$ starting from the bold $i$ in cell $x$ and ending at $j$ in the diagonal cell $x_H$. 
	Then for any cell $z$ of the chain containing $j$ or $j'$, the cell $z_{NW}$ contains $i$ or $i'$ respectively. 
	\end{enumerate}
\end{lemma}
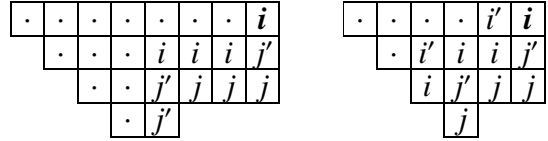
\Yboxdim 13pt
\begin{equation*}
\young(\cdot\cdot\cdot\cdot\cdot\cdot\cdot\boldi,:\cdot\cdot\cdot iii\jprime,::\cdot\cdot\jprime jjj,:::\cdot\jprime) \qquad
\young(\cdot\cdot\cdot\cdot\iprime\boldi,:\cdot\iprime ii\jprime,::i\jprime jj,:::j)
\end{equation*}

\begin{proof}
The proof of the first part is based on the observation that every $j$ in the chain must be bracketed with some $i$ in the 
reading word $\mathrm{rw}(\mathbf{T})$. 
Moreover, if the bold $i$ is located in row $r_x$ and rows $r_x, r_x+1,\ldots, r_z$ contain $n$ letters $j$,
then rows $r_x, r_x +1,\ldots, r_z-1$ must contain exactly $n$ non-bold letters $i$.
To prove that these elements $i$ must be located in the cells to the North-West of the cells containing $j$, we
proceed by induction on $n$. When we consider the next cell $z$ containing $j$ in the chain that must be bracketed, 
notice that the columns $c_z, c_z+1,\ldots, c_x$ already contain an $i$, and thus we must put the next $i$ in column 
$c_z -1$; there is no other row to put it than $r_z-1$. Thus, $z_{NW}$ must contain an $i$.

This line of logic also works for the second part of the lemma. We can show that for any cell $z$ of the chain containing 
$j$, the cell $z_{NW}$ must contain an $i$. As for cells $z$ containing $j'$, we can again use the fact that the corresponding 
letters $j$ in the primed reading word of $\mathbf{T}$ must be bracketed.
Notice that these letters $j'$ cannot be bracketed with unprimed letters $i$, since all unprimed letters $i$ are already bracketed 
with unprimed letters $j$. Thus, $j'$ must be bracketed with some $i'$ from a column to its left.
Let columns $1,2, \ldots, c_z$ contain $m$ elements $j'$.
Using the same induction argument as in the previous case, we can show that $z_{NW}$ must contain $i'$. 
\end{proof}

Next we need to figure out how $y$ in the raising crystal operator $e_i$ is related to the lowering operator rules
for $f_i$.
 
\begin{lemma}
\label{lemma.y}
Consider a pair of tableaux $\mathbf{T}$ and $\mathbf{T}' = f_i(\mathbf{T})$.
	\begin{enumerate}
	\item If tableau $\mathbf{T}$ (in case when bold $i$ in $\mathbf{T}$ is unprimed) or $\mathbf{T}^*$ 
	(if bold $i$ is primed) falls under Case (1) of the $f_i$ crystal operator rules, then cell $y$ of 
	the $e_i$ crystal operator rules is cell $x_E$ of $\mathbf{T}'$ or $(\mathbf{T}')^*$, respectively.
	
	\item If tableau $\mathbf{T}$ (in case when bold $i$ in $\mathbf{T}$ is unprimed) or $\mathbf{T}^*$ 
	(if bold $i$ is primed) falls under Case (2a) of the $f_i$ crystal operator rules, then cell $y$ of 
	the $e_i$ crystal operator rules is located in cell $x$ of $\mathbf{T}'$ or $(\mathbf{T}')^*$, respectively.
	
	\item If tableau $\mathbf{T}$ falls under Case (2b) of the $f_i$ crystal operator rules, then cell $y$ of 
	the $e_i$ crystal operator rules is cell $x^*$ of $(\mathbf{T}')^*$.
	
	\item If tableau $\mathbf{T}$ (in case when bold $i$ in $\mathbf{T}$ is unprimed) or $\mathbf{T}^*$ 
	(if bold $i$ is primed) falls under Case (2c) of the $f_i$ crystal operator rules, then cell $y$ of 
	the $e_i$ crystal operator rules is cell $x_H$ of $\mathbf{T}'$ or $(\mathbf{T}')^*$, respectively.
	\end{enumerate}
\end{lemma}

\begin{proof}
In all the cases above, we need to compare reading words $\mathrm{rw}(\mathbf{T})$ and $\mathrm{rw}(\mathbf{T}')$. 
Since $f_i$ affects at most two boxes of $\mathbf{T}$, it is easy to track how the reading word $\mathrm{rw}(\mathbf{T})$ 
changes after applying $f_i$. We want to check where the bold $j$ under $e_i$ ends up in $\mathrm{rw}(\mathbf{T}')$ 
and in $\mathbf{T}'$, which allows us to determine the cell $y$ of the $e_i$ crystal operator rules.

\smallskip

\noindent
\textbf{Case 1.1.} 
Suppose $\mathbf{T}$ falls under Case (1) of the $f_i$ crystal operator rules, that is, the bold $i$ in cell $x$ is to the left of 
$j'$ in cell $x_E$. Furthermore, $f_i$ acts on $\mathbf{T}$ by changing the entry in $x$ to $j'$ and by changing the entry
in $x_E$ to $j$. In the reading word $\mathrm{rw}(\mathbf{T})$, this corresponds to moving the $j$ corresponding to $x_E$ 
to the left and changing the bold $i$ (the rightmost unbracketed $i$) corresponding to cell $x$ to $j$ (that then corresponds
to $x_E$). Moving a bracketed $j$ in $\mathrm{rw}(\mathbf{T})$ to the left does not change the $\{i,j\}$ bracketing, and 
thus the $j$ corresponding to $x_E$ in $\mathrm{rw}(\mathbf{T}')$ is still the leftmost unbracketed $j$. Therefore, this $j$ 
is the bold $j$ of $\mathbf{T}'$ and is located in cell $x_E$.

\smallskip

\noindent
\textbf{Case 1.2.}
Suppose the bold $i$ in $\mathbf{T}$ is primed and $\mathbf{T}^*$ falls under Case (1) of the $f_i$ crystal operator 
rules. After applying lowering crystal operator rules to $\mathbf{T}^*$ and conjugating back, the bold primed $i$ in cell 
$x^*$ of $\mathbf{T}$ changes to an unprimed $i$, and the unprimed $i$ in cell $(x^*)_S$ of $\mathbf{T}$ changes 
to $j'$. In terms of the reading word of $\mathbf{T}$, it means moving the bracketed $i$ (in the unprimed reading word) 
corresponding to $(x^*)_S$ to the left so that it corresponds to $x^*$, and then changing the bold $i$ (in the 
primed reading word) corresponding to $x^*$ into the letter $j$ corresponding to $(x^*)_S$.
The first operation does not change the bracketing relations between $i$ and $j$, and thus the leftmost unbracketed $j$ 
in $\mathrm{rw}(\mathbf{T}')$ corresponds to $(x^*)_S$. Hence the bold unprimed $j$ is in cell $x_E$ of 
$(\mathbf{T}')^*$.

\smallskip

\noindent
\textbf{Case 2.1.}
If $\mathbf{T}$ falls under Case (2a) of the $f_i$ crystal operator rules, $f_i$ just changes the entry in $x$ from $i$ to $j$.
The rightmost unbracketed $i$ in the reading word of $\mathbf{T}$ changes to the leftmost unbracketed $j$ in 
$\mathrm{rw}(\mathbf{T}')$. Thus, the bold $j$ in $\mathrm{rw}(\mathbf{T}')$ corresponds to cell $x$.

\smallskip

\noindent
\textbf{Case 2.2.}
The case when $\mathbf{T}^*$ falls under Case (2a) of the $f_i$ crystal operator rules is the same as the previous case.

\smallskip

\noindent
\textbf{Case 3.}
Suppose $\mathbf{T}$ falls under Case (2b) of $f_i$ crystal operator rules. Then there is a chain starting from cell 
$x$ (of row $r_x$ and column $c_x$) and ending at the diagonal cell $z$ (of row and column $r_z$) consisting of 
elements $j$ and $j'$. Applying $f_i$ to $\mathbf{T}$ changes the entry in $x$ from $i$ to $j'$.
In $\mathrm{rw}(\mathbf{T})$ this implies moving the bold $i$ from the unprimed reading word to the left through 
elements $i$ and $j$ corresponding to rows $r_x, r_x +1,\ldots, r_z$, then through elements $i$ and $j$ in the primed 
reading word corresponding to columns $c_z-1, \ldots, c_x$, and then changing that $i$ to $j$ which corresponds to 
cell $x$. But according to Lemma~\ref{lemma.chains}, the letters $i$ and $j$ in these rows and columns are all bracketed 
with each other, since for every $j$ or $j'$ in the chain there is a corresponding $i$ or $i'$ in the North-Western cell. 
(Notice that there cannot be any other letter $j$ or $j'$ outside of the chain in rows $r_x +1,\ldots, r_z$ and in columns 
$c_z-1, \ldots, c_x$.) Thus, moving the bold $i$ to the left in $\mathrm{rw}(\mathbf{T})$ does not change the bracketing 
relations. Changing it to $j$ makes it the leftmost unbracketed $j$ in $\mathrm{rw}(\mathbf{T}')$.
Therefore, the bold $j$ in $\mathrm{rw}(\mathbf{T}')$ corresponds to the primed $j$ in cell $x$ of $\mathbf{T}'$, 
and the cell $y$ of the $e_i$ crystal operator rules is thus cell $x^*$ in $(\mathbf{T}')^*$.

\smallskip

\noindent
\textbf{Case 4.1.}
Suppose $\mathbf{T}$ falls under Case (2c) of the $f_i$ crystal operator rules. There is a chain starting from cell $x$ 
(in row $r_x$ and column $c_x$) and ending at cell $x_H$ (in row $r_H$ and column $c_H$) consisting of elements 
$j$ and $j'$. Applying $f_i$ to $\mathbf{T}$ changes the entry in $x$ from $i$ to $j'$ and changes the entry in 
$x_H$ from $j'$ to $j$. Moving $j'$ from cell $x_H$ to cell $x$ moves the corresponding bracketed $j$ in the reading 
word $\mathrm{rw}(\mathbf{T})$ to the left, and thus does not change the $\{i,j\}$ bracketing relations in 
$\mathrm{rw}(\mathbf{T}')$. On the other hand, moving the bold $i$ from cell $x$ to cell $x_H$ and then changing it to 
$j$ moves the bold $i$ in $\mathrm{rw}(\mathbf{T})$ to the right through elements $i$ and $j$ corresponding to rows 
$r_x, r_x +1,\ldots, r_H$, and then changes it to $j$. Note that according to Lemma~\ref{lemma.chains}, each $j$ in 
rows $r_x+1, r_x +2,\ldots, r_H$ has a corresponding $i$ from rows $r_x, r_x +1,\ldots, r_H - 1$ that it is
bracketed with, and vise versa. Thus, moving the bold $i$ to the position corresponding to $x_H$ does not change the 
fact that it is the rightmost unbracketed $i$ in $\mathrm{rw}(\mathbf{T})$.
Thus, the bold $j$ in $\mathrm{rw}(\mathbf{T}')$ corresponds to the unprimed $j$ in cell $x_H$ of $\mathbf{T}'$.

\smallskip

\noindent
\textbf{Case 4.2.}
Suppose $\mathbf{T}$ has a primed bold $i$ and $\mathbf{T}^*$ falls under Case (2c) of the $f_i$ crystal operator rules.
This means that there is a chain (expanding in North and East directions) in $\mathbf{T}$ starting from $i'$ in cell $x^*$
and ending in cell $x_H^*$ with entry $i$ consisting of elements $i$ and $j'$. The crystal operator $f_i$ changes 
the entry in cell $x^*$ from $i'$ to $i$ and changes the entry in $x_H^*$ from $i$ to $j'$. For the reading word 
$\mathrm{rw}(\mathbf{T})$ this means moving the bracketed $i$ in the unprimed reading word to the right 
(which does not change the bracketing relations) and moving the bold $i$ in the primed reading word through letters
$i$ and $j$ corresponding to columns $c_x, c_x +1 ,\ldots, c_H$, which are bracketed with each other according to 
Lemma~\ref{lemma.chains}.
Thus, after changing the bold $i$ to $j$ makes it the leftmost unbracketed $j$ in $\mathrm{rw}(\mathbf{T}')$. Hence the
bold primed $j$ in $\mathbf{T}'$ corresponds to cell $x_H^*$.
Therefore $y$ from the $e_i$ crystal operator rules is cell $x_H$ of $(\mathbf{T}')^*$.
\end{proof}

\subsection{Proof of Theorem~\ref{theorem.main3}}
Let $\mathbf{T'} = f_i(\mathbf{T})$.

\smallskip

\noindent
\textbf{Case 1.} 
If $\mathbf{T}$ (or $\mathbf{T}^*$) falls under Case (1) of the $f_i$ crystal operator rules, then according to 
Lemma~\ref{lemma.y}, $e_i$ acts on $\mathbf{T}'$ (or on $(\mathbf{T}')^*$) by changing the entry in cell 
$y_W = x$ back to $i$ and changing the entry in $y = x_E$ back to $j'$. Thus, the statement of the theorem is true.

\smallskip

\noindent
\textbf{Case 2.} 
If $\mathbf{T}$ (or $\mathbf{T}^*$) falls under Case (2a) of the $f_i$ crystal operator rules, then according to 
Lemma~\ref{lemma.y}, $e_i$ acts on $\mathbf{T}'$ (or on $(\mathbf{T}')^*$) by changing the entry in the cell 
$y = x$ back to $i$. Thus, the statement of the theorem is true.

\smallskip

\noindent
\textbf{Case 3.} If $\mathbf{T}$ falls under Case (2b) of the $f_i$ crystal operator rules, then according to 
Lemma~\ref{lemma.y}, $e_i$ acts on cell $y=x^*$ of $(\mathbf{T}')^*$. Note that according to Lemma~\ref{lemma.chains}, 
there is a maximal chain of letters $i$ and $j'$ in $(\mathbf{T}')^*$ starting at $y$ and ending at a diagonal cell $y_T$. 
Thus, $e_i$ changes the entry in cell $y=x^*$ in $(\mathbf{T}')^*$ from $j$ to $j'$, so the entry in cell $x$ in 
$\mathbf{T}'$ goes back from $j'$ to $i$. Thus, the statement of the theorem is true.

\smallskip

\noindent
\textbf{Case 4.} If $\mathbf{T}$ (or $\mathbf{T}^*$) falls under Case (2c) of the $f_i$ crystal operator rules, then according 
to Lemma~\ref{lemma.y}, $e_i$ acts on cell $y=x_H$ of $\mathbf{T}'$ (or of $(\mathbf{T}')^*$). Note that according 
to Lemma~\ref{lemma.chains}, there is a maximal (since $c(x_E) \neq j'$ and $c(x_E) \neq i$) chain of letters $i$ and $j'$
in $\mathbf{T}'$ (or $(\mathbf{T}')^*$) starting at $y$ and ending at cell $y_T = x$.
Thus, $e_i$ changes the entry in cell $y=x_H$ in $(\mathbf{T}')^*$ from $j$ back to $j'$ and changes the entry in 
$y_T = x$ from $j'$ back to $i$. Thus, the statement of the theorem is true.

\bibliographystyle{plain}


\end{document}